\RequirePackage{fix-cm}
\RequirePackage{amsmath}
\RequirePackage{amsfonts}
\documentclass[smallextended,numbook,runningheads]{svjour3}     % onecolumn (second format)
\usepackage[T1]{fontenc}
\usepackage{graphicx}
\usepackage{times}    % use Times fonts if available on your TeX system
\usepackage{mathrsfs}
\usepackage{amsfonts}
\usepackage{amssymb}
\usepackage{graphicx,subfigure,color}
\usepackage{algorithm}
\usepackage{multirow}
\usepackage{xcolor}
\usepackage{setspace}
\usepackage[numbers,sort&compress]{natbib}
\usepackage{subfigure}
\usepackage{booktabs}
\usepackage{array}
\newcommand{\diag}{\mathrm{diag}}
\newcommand{\cond}{\rm cond}

\makeatletter
\def\@textbottom{\vskip \z@ \@plus 100pt}
\let\@texttop\relax
\makeatother
\journalname{BIT}
\begin{document}

\title{An Efficient Second-Order Convergent Scheme for One-Side Space Fractional Diffusion Equations with Variable Coefficients\thanks{This research was supported by research grants, 12200317, 12306616, 12302715, 12301214 from HKRGC GRF, MYRG2016-00063-FST from University of Macau and 054/2015/A2 from FDCT of Macao, Macao Science and Technology Development Fund 010/2015/A, 050/2017/A and the Grant MYRG2017-00098-FST from University of Macau.
}}

%\titlerunning{Short form of title}        % if too long for running head

\author{Xue-lei Lin     \and Pin Lyu \and
        Michael K. Ng \and Hai-Wei Sun\and Seakweng Vong  
}

%\authorrunning{Short form of author list} % if too long for running head

\institute{Xue-Lei Lin \at
              Department of Mathematics, Hong Kong Baptist University. \\
              \email{hxuellin@gmail.com}           %  \\
%             \emph{Present address:} of F. Author  %  if needed
           \and
           Pin Lyu \at 
           School of Economic Mathematics, Southwestern University of Finance and Economics, Chengdu, China.\\
            Department of Mathematics, University of Macau, Macao.\\
           \email{lyupin1991@163.com}
           \and
           Michael K. Ng \at
              Department of Mathematics, Hong Kong Baptist University, Hong Kong.\\
              \email{mng@math.hkbu.edu.hk}
           \and
           Hai-Wei Sun \at
           Department of Mathematics, University of Macau, Macao.\\
           \email{hsun@umac.mo}
            \and
           Seakweng Vong \at
           Department of Mathematics, University of Macau, Macao.\\
           \email{swvong@umac.mo}   
}

%\date{Received: date / Accepted: date}
% The correct dates will be entered by the editor

\maketitle

\begin{abstract}
	In this paper, a second order finite difference scheme is investigated for time-dependent one-side space fractional diffusion equations with variable coefficients. The existing schemes for the equation with variable coefficients have temporal convergence rate no better than second order and spatial convergence rate no better than first order, theoretically. 
	In the presented scheme, the Crank-Nicolson temporal discretization and a second-order weighted-and-shifted Gr\"{u}nwald-Letnikov spatial discretization are employed. Theoretically, the unconditional stability and the second-order convergence in time and space of the proposed scheme are established under some conditions on the diffusion coefficients. Moreover, a Toeplitz preconditioner is proposed for linear systems arising from the proposed scheme. The condition number of the preconditioned matrix is proven to be bounded by a constant independent of the discretization step-sizes so that the Krylov subspace solver for the preconditioned linear systems converges linearly. Numerical results are reported to show the convergence rate and the efficiency of the proposed scheme.
\keywords{one-side space-fractional diffusion equation \and variable diffusion coefficients \and stability and convergence\and high-order finite difference scheme \and preconditioner}
% \PACS{PACS code1 \and PACS code2 \and more}
\subclass{26A33 \and 35R11 \and 65M06\and 65M12}
\end{abstract}

	\section{Introduction}\label{introduction}
 In the paper, we study an efficient numerical method for solving the one-side space fractional diffusion equation (OSFDE) with variable coefficients. To begin with, we firstly present  the one-dimensional OSFDE (the two dimension case will be discussed in Section 3) \cite{tadjeran2006205,sousa2011numerical,sousaelic}:
\begin{equation}\label{1D-fde}
\begin{array}{l}
\displaystyle{\frac{\partial u(x,t)}{\partial t}=
	d(x)\,{_{x_L}}D_x^{\alpha}u(x,t)+f(x,t)}, \qquad x\in(x_L,x_R),~t\in(0,T],\\
\vspace{2mm}
u(x_L,t)=u(x_R,t)=0, \qquad  t\in[0,T],\\
u(x,0)=\varphi(x),\qquad x\in[x_L,x_R],
\end{array}
\end{equation}
where $d(x)$, which satisfies $0< {d_-}\leq d(x)\leq {d_+} <\infty$, is a strictly positive known function, $\varphi$ and $f$ are both known functions, $u$ is unknown to be solved, ${_{x_L}}D_x^{\alpha}u(x,t)$ is the Riemann-Liouville (RL) fractional derivative of order $\alpha\in (1,2)$ defined as (\cite{podlubny1999,samko1993fractional})
\begin{equation}\label{RL}
\begin{array}{l}\vspace{2mm}
\displaystyle{\,{_{x_L}}D_x^{\alpha}u(x,t)=\dfrac{1}{\Gamma(2-\alpha)}\dfrac{\partial^2}{\partial
		x^2}\int_{x_L}^x \dfrac{u(\xi,t)}{(x-\xi)^{\alpha-1}}\,d\xi},
\end{array}
\end{equation}
with $\Gamma(\cdot)$ denoting the gamma function.

Due to the nonlocal dependence, fractional derivatives model many challenging phenomena more accurately than integer-order derivatives do, which has therefore attracted lots of interests in recent years. As an illustration of this fact, some applications of fractional calculus and anomalous diffusion have been discussed in the books \cite{baleanu2010new,klages2008anomalous,klafter2012fractional,mainardi2010fractional,sabatier2007advances,ortigueira2011fractional} and the references therein. It is well-known that closed-form analytic solutions of fractional diffusion equations are usually not available especially in the existence of variable coefficients. Moreover, because of the nonlocality of the fractional derivative and the existence of the variable coefficients, the discretization of the OSFDE tends to generate dense matrix with high displacement rank \footnote{see \cite{mng2004}}, for which the discrete linear systems related to the variable-coefficients OSFDE are time-consuming to directly solve. Therefore, studying reliable discretization schemes and the corresponding fast iterative solvers for the OSFDE becomes an urgent topic.

There have been many schemes applicable to or solely developed for OSFDEs; see, e.g., \cite{tadjeran2006205,meerschaert2004finite,tadjeran2007813,tianwy2015,chendeng2014,sousa2011numerical,
	sousaelic,qu2014note,vong2016high,zhaozsunwcao2015,leislhyc2016}. In \cite{zhaozsunwcao2015,leislhyc2016,vong2016high}, numerical schemes with spatial fourth-order convergence for a space fractional diffusion equation is developed by applying the technique of compact operators, which is however only available for constant-coefficient case. Another spatially fourth-order accurate scheme is studied in \cite{chendeng2014} by implementing weighted-and-shifted Lubich difference operators whose convergent property is established only for constant diffusion coefficients.
Some second-order numerical schemes are proposed in \cite{tadjeran2006205,tadjeran2007813,sousa2011numerical,sousaelic} for solving OSFDEs with variable coefficients,  which however does not provide convergence proof.
In \cite{qu2014note}, the stability and convergence of the second-order numerical scheme for variable coefficient equations
are established for
	$\alpha\in(\alpha_0,2)$, where $\alpha_0\approx 1.5546$ is a solution of the equation $3^{3-\gamma}-4\times2^{3-\gamma}+6=0$.
	In \cite{linstbcvg2017}, a series of numerical schemes for Riesz space fractional diffusion equation have been proven to be convergent and stable. Nevertheless, the proof technique used in \cite{linstbcvg2017} heavily depends on the symmetry of discretization matrix of Riesz fractional derivative, which is not applicable to the OSFDE that involves the non-symmetric one-sided fractional derivatives weighted by variable coefficients.

In this paper, we propose a second-order scheme for the one- and two-dimensional OSFDEs weighted by variable coefficients. The Crank-Nicolson method and a second-order weighted-and-shifted Gr\"{u}nwald-Letnikov difference (WSGD, see \cite{tianwy2015}) operator are employed to discretize temporal and spatial derivatives, respectively. For the one-dimensional OSFDE, the proposed scheme is proven to be unconditionally stable and second-order convergent in time and space without additional assumption on the diffusion coefficient.
It has been shown in \cite{vonglyu2017} that the symmetric part of the discretization matrix of $d(x){_{x_L}}D_x^{\alpha}$ is negative semi-definite under some conditions on $d(x)$. We extend this one-dimension result to two-dimension case under additional assumptions on the diffusion coefficients \footnote{see the assumptions in Lemma \ref{case2lemma}}, based on which the proposed scheme for the two-dimension OSFDE is proven to be unconditionally stable and second-order convergent in time and space.

As mentioned above, the direct solver for the linear systems arising from variable-coefficients OSFDE requires too much computational time. Fortunately, the discretization matrix has Toeplitz-like structure due to which its matrix-vector multiplication can be fast computed via fast Fourier transforms (FFTs). Because of the fast matrix-vector multiplication, fast iterative solvers for the linear systems can be possibly developed. However, the discretization matrix of the OSFDE is ill-conditioned when $\tau/h^{\alpha}$ is large, where $\tau$ and $h$ represent the temporal and spatial step-sizes, respectively. Thus, a Toeplitz preconditioner is proposed to reduce the condition number of the one- and two-dimensional discretization matrices. Theoretically, we show that the condition number of the preconditioned matrix is uniformly bounded by a constant independent of $\tau$ and $h$ under certain conditions on the diffusion coefficients \footnote{see the assumptions in Theorems \ref{1dcondesti}, \ref{2dclstthm}} so that the Krylov subspace method for the preconditioned linear systems converges linearly no matter the unpreconditioned matrix is ill-conditioned or not.

This paper is organized as follows. In Section 2, we propose a second-order scheme and its corresponding Toeplitz preconditioner for the one-dimensional \\OSFDE, analyze the unconditional stability and convergence of the proposed scheme, estimate the condition number of the preconditioned matrix. In Section 3, we extend the scheme and the preconditioner to two-dimensional case. In Section 4, numerical results are reported to show the efficiency and accuracy of the proposed scheme.

\section{Stability and Convergence of Discrete One-Dimensional OSFDE and Its Preconditioning}\label{1D}
%We first consider the following one-dimensional space fractional diffusion equation \cite{tadjeran2006205,sousa2011numerical,sousaelic}:

We need some notations to describe the discretization for \eqref{1D-fde}. Let $h = (x_R-x_L)/(M+1)$ and $\tau =
T/N$ be the space and time step sizes, respectively, where $M$ and $N$ are given positive integers. And denote $x_i = x_L + i h$ for $i=0,1,\ldots,M+1$, $t_n = n\tau$ for $n = 0,1,\ldots,N$. Throughout this paper, the discretization on RL fractional derivative is based on the following second-order WSGD formula \cite{tianwy2015}:
\begin{equation}\label{FD2nd}
\begin{array}{l} \vspace{2mm}
\displaystyle{{_{x_L}}D_x^{\alpha}u(x_i)=\frac{1}{h^{\alpha}}\sum_{k=0}^{i}w_{k}^{(\alpha)}u(x_{i-k+1})+{\cal O}(h^2)},
\end{array}
\end{equation}
which is under the smooth assumptions $u$, $_{-\infty}D_x^{\alpha+2}u$ and Fourier transform of\\ $_{-\infty}D_x^{\alpha+2}u$ belong to $L^1(\mathbb{R})$. The coefficients $w_{k}^{(\alpha)}$ were defined by (\cite{tianwy2015})
\begin{equation}\label{w_{k}(1,0)}
w_0^{(\alpha)}=\frac{\alpha}{2}g_0^{(\alpha)},
~w_k^{(\alpha)}=\frac{\alpha}{2}g_k^{(\alpha)}+\frac{2-\alpha}{2}g_{k-1}^{(\alpha)}~\mbox{  for  }~k\geq1,
\end{equation}
where $g_k^{(\alpha)}$ are the coefficients of the power series of $(1-z)^\alpha$, and they can be obtained recursively as
\begin{align*}%\label{gkalpha}
g_0^{(\alpha)}=1,\quad g_k^{(\alpha)}=\Big(1-\frac{\alpha+1}{k} \Big)g_{k-1}^{(\alpha)},~\mbox{ for }~k=1,2,\ldots.
\end{align*}

Next, we introduce the finite difference scheme for solving \eqref{1D-fde}. Let $u^{n}_i$ be the numerical approximation of $u(x_i,t_n)$, and denote $d_i=d(x_i)$, $\varphi_i=\varphi(x_i)$, $f^{n-\frac{1}{2}}_i=f(x_i,t_{n-\frac{1}{2}})$, where $t_{n-\frac{1}{2}}=(t_{n-1}+t_n)/2$ and $n=1,2,\ldots,N$.
Then, applying the Crank-Nicolson technique and approximation
\eqref{FD2nd} to the time derivative and the space fractional derivatives of \eqref{1D-fde} respectively, we get
\begin{align}\label{CN}
\frac{u_{i}^{n}-u_{i}^{n-1}}{\tau}=
\frac1{2h^{\alpha}}d_i\sum_{k=0}^{i}w_{k}^{(\alpha)}\left(u_{i-k+1}^{n-1}
+u_{i-k+1}^n\right)
+ f_{i}^{n-\frac{1}{2}}+R_i^{n-\frac12},~ 1\leq i\leq M,
\end{align}
where $|R_i^{n-\frac12}|\leq c_1(\tau^2+h^2)$ for a positive constant $c_1$; see, e.g., \cite{tianwy2015}.

Denote $u^n=[u_1^n,u_2^n,\ldots,u_{M}^n]^T$,
$f^{n-\frac{1}{2}}=[f_1^{n-\frac{1}{2}},f_2^{n-\frac{1}{2}},\ldots,f_{M}^{n-\frac{1}{2}}]^T$, and
\begin{eqnarray}\label{G-alpha}
D=\diag(d_1,d_2,...,d_M),
\quad
G_\alpha=\left[
\begin{array}{ccccc}
w_1^{(\alpha)}& w_0^{(\alpha)} & 0 & \cdots &  0 \\
w_2^{(\alpha)} & w_1^{(\alpha)} & w_0^{(\alpha)} & \ddots & \vdots \\
\vdots & w_2^{(\alpha)} & w_1^{(\alpha)} & \ddots & 0\\
\vdots & \ddots & \ddots & \ddots & w_0^{(\alpha)} \\
w_{M}^{(\alpha)}  & \cdots  & \cdots & w_2^{(\alpha)} & w_1^{(\alpha)} \\
\end{array}
\right],
\end{eqnarray}
where  $\{w_k^{(\alpha)}\}_{k=0}^{M}$ are the coefficients given in \eqref{w_{k}(1,0)}.

Omitting the small term $R_i^{n+\frac12}$ in \eqref{CN}, then equation \eqref{1D-fde} can be solved numerically by the following finite difference scheme in matrix form
\begin{equation}\label{matrixform}
\frac1{\tau}\left(u^{n}-u^{n-1}\right)=\frac{1}{2h^\alpha}DG_\alpha\left(u^{n-1}+u^{n} \right)+ f^{n-\frac{1}{2}},
\quad n=1,2,\ldots,N.
\end{equation}
\subsection{Stability and convergence}
\emph{Some General Notations:}

$\bullet$  $\mathbb{C}^{m\times n}$ ($\mathbb{R}^{m\times n}$, respectively) denotes the set of all $m\times n$ complex (real, respectively) matrices.

$\bullet$  ${\cal H}(X)$ denotes the symmetric part of a square matrix $X$.

\begin{lemma}(\cite{tianwy2015})\label{Galpha}
	The matrix $G_\alpha+G_\alpha^T$ is negative definite.
\end{lemma}
\begin{lemma}(\cite{laub2005matrix})\label{bilinear_inequality}
	Let symmetric matrix $H\in \mathbb{R}^{n\times n}$ with eigenvalues
	$\lambda_1\ge \lambda_2\ge \ldots \ge \lambda_n$. Then for all $w\in
	\mathbb{R}^{n\times 1}$,
	$$\lambda_n w^Tw\leq w^THw\leq \lambda_1 w^Tw.$$
\end{lemma}
Now we show the stability and convergence of scheme \eqref{matrixform} by energy method.
\begin{theorem}\label{stable-1D}
	The finite difference scheme \eqref{matrixform} is unconditionally stable and its solution satisfies the following estimate
		$$\left\|u^{n}\right\|_{D^{-1}}^2\leq \exp(2T)\left\|\varphi\right\|_{D^{-1}}^2+[\exp(2T)-1]\max_{1\leq k\leq n}\left\|f^{k-\frac12}\right\|_{D^{-1}}^2 ,~ n=1,2,\ldots,N,$$
	where $\|\cdot\|_{D^{-1}}$ is the norm induced by the inner product, $\langle v_1,v_2\rangle_{D^{-1}}:=hv_1^{\rm T}D^{-1}v_2$.
\end{theorem}
\begin{proof}
	Some steps of this proof are similar to those of Theorem 3.8 in \cite{vonglyu2017}.
	Multiplying \\
	$h\left(u^{n-1}+u^{n} \right)^TD^{-1}$ on the both sides of \eqref{matrixform}, we get
	\begin{align}\nonumber
	\frac1{\tau}h\left(u^{n-1}+u^{n}\right)^TD^{-1}\left(u^{n}-u^{n-1}\right)=&\frac{1}{2h^\alpha}h\left(u^{n-1}+u^{n} \right)^TG_\alpha\left(u^{n-1}+u^{n} \right)\\\label{stble-proof1}
	&+ h\left(u^{n-1}+u^{n} \right)^TD^{-1}f^{n-\frac{1}{2}}.
	\end{align}
	Notice that $w^TG_\alpha w=w^T{\cal H}(G_\alpha)w$ for any real vector $w$.
	Therefore, by Lemma \ref{Galpha}, the first term on the right hand side of \eqref{stble-proof1} can be estimated as
	\begin{align*}
	&\frac{1}{2h^\alpha}h\left(u^{n-1}+u^{n} \right)^TG_\alpha\left(u^{n-1}+u^{n} \right)\\
	&=\frac{1}{2h^\alpha}h\left(u^{n-1}+u^{n} \right)^T{\cal H}(G_\alpha)\left(u^{n-1}+u^{n} \right)\leq0.
	\end{align*}
	As a result
	\begin{align}\label{stable1D4}
	&h(u^{n})^TD^{-1}u^{n}- h(u^{n-1})^TD^{-1}u^{n-1}\notag\\
	&\le\tau h(u^{n})^T D^{-1}f^{n-\frac{1}{2}}
	+\tau h(u^{n-1})^T D^{-1}f^{n-\frac{1}{2}}.
	\end{align}
	Applying Cauchy-Schwarz inequality on the right hand side of \eqref{stable1D4}, we get
	\begin{align*}
	\left\|u^{n}\right\|_{D^{-1}}^2\leq\left\|u^{n-1}\right\|_{D^{-1}}^2+\frac{\tau}{2}\left\|u^{n}\right\|_{D^{-1}}^2
	+\frac{\tau}{2}\left\|u^{n-1}\right\|_{D^{-1}}^2+\tau\left\|f^{n-\frac12}\right\|_{D^{-1}}^2,
	\end{align*}
	which is equivalent to
	\begin{align}\label{stable1D5}
	\left\|u^{n}\right\|_{D^{-1}}^2\leq\frac{2+\tau}{2-\tau}\left\|u^{n-1}\right\|_{D^{-1}}^2+
	\frac{2\tau}{2-\tau}\left\|f^{n-\frac12}\right\|_{D^{-1}}^2.
	\end{align}
	Iterating \eqref{stable1D5} for $n$ times, we obtain
	\begin{align}\nonumber
	\left\|u^{n}\right\|_{D^{-1}}^2\leq&\Big(\frac{2+\tau}{2-\tau}\Big)^{n}\left\|u^{0}\right\|_{D^{-1}}^2\\\label{stable1D6}
	+\frac{2\tau}{2-\tau}&\Big[1+\frac{2+\tau}{2-\tau}+\Big(\frac{2+\tau}{2-\tau}\Big)^2+\ldots+
	\Big(\frac{2+\tau}{2-\tau}\Big)^{n-1}\Big]\max_{1\leq k\leq n}\left\|f^{k-\frac12}\right\|_{D^{-1}}^2.
	\end{align}
	For the small $\tau~(\tau\leq1)$, we have
	\begin{align}\label{stable1D7}
	\Big(\frac{2+\tau}{2-\tau}\Big)^{n}=\Big(1+\frac{2\tau}{2-\tau}\Big)^{n}\leq(1+2\tau)^{n}\leq\lim\limits_{N\rightarrow+\infty}\left(1+\frac{2T}{N} \right)^N=\exp(2T),
	\end{align}
	and
		\begin{align}\label{stable1D8}
		\frac{2\tau}{2-\tau}\sum_{k=1}^{n}\Big(\frac{2+\tau}{2-\tau}\Big)^{k-1}=\left(\frac{2+\tau}{2-\tau}\right)^n-1\leq \exp(2T)-1 .
		\end{align}
	The result follows from \eqref{stable1D6}--\eqref{stable1D8}.
\end{proof}

\begin{theorem}\label{convergence1D}
	Let $u(x_i,t_n)$ be the exact solution of \eqref{1D-fde} and $u_i^n$ be the solution of finite difference scheme \eqref{matrixform}. Denote $e_i^n=u(x_i,t_n)-u_i^n$, $0\leq i\leq M+1$, $0\leq n\leq N$. Then there exists a positive constant $c_2$ such that
	$$\|e^n\|\leq c_2(\tau^2+h^2),$$
	where $e^n=[e_1^n,e_2^n,\ldots,e_{M}^n]^T$ and $\|\cdot\|$ denotes the discrete $L^2$ norm, i.e. $\|v\|=\sqrt{hv^Tv}$, 
\end{theorem}
\begin{proof}
	Denote $R^{n-\frac{1}{2}}=[R_1^{n-\frac{1}{2}},R_2^{n-\frac{1}{2}},\ldots,R_{M}^{n-\frac{1}{2}}]^T$.
	We can easily show that $e^n$ and $e_i^n$ satisfy the following error equations
	\begin{align*}
	&\frac1{\tau}\left(e^{n}-e^{n-1}\right)=\frac{1}{2h^\alpha}DG_\alpha\left(e^{n-1}+e^{n} \right)+ R^{n-\frac{1}{2}}, \quad 1\leq n\leq N,\\
	& e_0^n=e_{M+1}^n=0, \quad 1\leq n\leq N,  \qquad e_i^0=0, \quad 0\leq i\leq M+1.
	\end{align*}
	By Theorem \ref{stable-1D}, we have
	$$\left\|e^{n}\right\|_{D^{-1}}^2\leq [\exp(2T)-1]\max_{1\leq k\leq n}\left\|R^{k-\frac12}\right\|_{D^{-1}}^2 ,\quad n=1,2,\ldots,N.$$
	As $D^{-1}$ is a positive diagonal matrix, utilizing Lemma \ref{bilinear_inequality}, we get
	$$\|e^{n}\|^2\leq [c_2(\tau^2+h^2)]^2,\quad n=1,2,\ldots,N.$$
\end{proof}

\subsection{An estimate on the field of values of $DG_\alpha+G_\alpha^TD$}
In this subsection, we focus on estimating the field of values of $DG_\alpha+G_\alpha^TD$, the results of which will be further applied to the analysis of one-dimensional preconditioning and the extension to two-dimensional OSFDE. First, we denote $g(\alpha,x)$ as the generating function \cite{mng2004} of the Toeplitz matrix $G_\alpha$. The next two lemmas describe some properties concerning $g(\alpha,x)$, which will be useful to obtaining the desired estimation.
\begin{lemma}(\cite{mng2004})\label{JIN}
	Let ${\bf u}=[u_1,u_2,\ldots,u_{M}]^T,{\bf v}=[v_1,v_2,\ldots,v_{M}]^T\in \mathbb{R}^{M\times 1}$. Then we have
	$${\bf u}^TG_\alpha {\bf v}=\frac{1}{2\pi}\int_{-\pi}^\pi\sum_{k=1}^{M}{ u}_ke^{-{\bf i}kx}\sum_{k=1}^{M}{ v}_ke^{{\bf i}kx}g(\alpha,x)dx.$$
\end{lemma}
\begin{lemma}(\cite{vonglyu2017})\label{generatingfunction_rario}
	It holds that
	$$\varsigma_\alpha\triangleq\min_{x}\frac{\Re[-g(\alpha,x)]}{|g(\alpha,x)|}=|\cos(\frac{\alpha}{2}\pi)|,$$
	where $\Re[g(\alpha,x)]$ denotes the real part of $g(\alpha,x)$.
\end{lemma}
The following lemma provides a novel bound to the field of values of
${\tilde D}G{\tilde D}$, where $G=-G_\alpha-G_\alpha^T$ and ${\tilde D}$ is a diagonal matrix satisfying some properties.
\begin{lemma}(\cite{vonglyu2017})\label{key_inequality1}
	Denote $G=-G_\alpha-G_\alpha^T$.
	Suppose that\\ ${\tilde D}=\diag(\tilde{d}(x_1),\tilde{d}(x_2),...,\tilde{d}(x_M))$ for some function $\tilde{d}(x)$ defined on $(x_L,x_R)$.
	For any real vector ${\bf u}=[u_1,u_2,\ldots,u_{M}]^T$, we have
	\begin{align}\label{key_inequality}
	{\bf u}^T{\tilde D}G{\tilde D}{\bf u}\leq 2\max_i\{|{\tilde d}_i|^2\}{\bf u}^TG{\bf u},
	\end{align}
	if ${\tilde d}(x)$ is convex and ${\tilde d}(x)\geq0$, or ${\tilde d}(x)$ is concave and ${\tilde d}(x)\leq0$.
\end{lemma}
Assuming $0\leq \kappa_{min}\leq d(x)\leq\kappa_{max}<\infty$. The following theorem reveals some inclusion relations between numerical ranges of $G$ and $-DG_{\alpha}-G_{\alpha}^{\rm T}D$, which acts an important role in the analysis of the proposed preconditioner.
\begin{theorem}\label{key-estimate}
	For any ${\bf u}=[u_1,u_2,\ldots,u_{M}]^T$, we have
	\begin{align}\label{cond}
	&\left(\kappa-\frac{\sqrt{2}(\kappa_{max}-\kappa_{min})}{\varsigma_\alpha}\right){\bf u}^TG{\bf u}\notag\\
	&\leq {\bf u}^T(-DG_\alpha-G_\alpha^TD){\bf u}\leq \left(\kappa+\frac{\sqrt{2}(\kappa_{max}-\kappa_{min})}{\varsigma_\alpha}\right){\bf u}^TG{\bf u},
	\end{align}
	where $\kappa=\kappa_{max}$ when $d(x)$ is concave, and $\kappa=\kappa_{min}$ when $d(x)$ is convex.
\end{theorem}
\begin{proof}
	Denote ${\tilde D}=D-\kappa I$, then $DG_\alpha+G_\alpha^TD=\kappa(G_\alpha+G_\alpha^T)+{\tilde D}G_\alpha+G_\alpha^T{\tilde D}$.
	And, for any ${\bf u}=[u_1,u_2,\ldots,u_{M}]^T$, we have
	\begin{align*}
	{\bf u}^T(-DG_\alpha-G_\alpha^TD){\bf u}=\kappa{\bf u}^TG{\bf u}
	+{\bf u}^T\left(-{\tilde D}G_\alpha-G_\alpha^T {\tilde D}\right){\bf u}.
	\end{align*}
	Denote $u(x)=\sum_{k=1}^{M}{ u}_ke^{{\bf i}kx}$
	and $v(x)=\sum_{k=1}^{M}{({\tilde D}u)}_ke^{{\bf i}kx}$, it follows by Lemmas \ref{JIN}
	and \ref{key_inequality1} that
	\begin{align}\label{condproof2}
	{\bf  u}^TG{\bf  u}&={\bf  u}^T\left(-G_\alpha-G_\alpha^T\right){\bf  u}
	=\frac1{\pi}\int_{-\pi}^\pi
	\Re[-g(\alpha,x)]|u(x)|^2dx,\\\label{condproof2-2}
	 {\bf  u}^T{\tilde D}W{\tilde D}{\bf  u}
	&=\frac1{\pi}\int_{-\pi}^\pi
	\Re[-g(\alpha,x)]|v(x)|^2dx\notag\\
	&\leq \frac{2(\kappa_{max}-\kappa_{min})^2}{\pi}\int_{-\pi}^\pi
	\Re[-g(\alpha,x)]|u(x)|^2dx.
	\end{align}
	Using Lemma \ref{JIN} again and applying Cauchy-Schwarz inequality, Lemma \ref{generatingfunction_rario}, \eqref{condproof2} and \eqref{condproof2-2}, we get
	\begin{align*}
	&\left|{\bf  u}^T\left(-{\tilde D}G_\alpha-G_\alpha^T {\tilde D}\right){\bf  u}\right|\notag\\
	&=\frac1{2\pi}\Big|\int_{-\pi}^\pi
	\big(-v^*gu-u^*g^*v\big) dx\Big|\\
	&{\le}\frac1{\pi}\int_{-\pi}^\pi
	|g(\alpha,x)||v(x)||u(x)| dx\\
	&\leq \frac1{\pi\varsigma_\alpha}\int_{-\pi}^\pi
	\Re[-g(\alpha,x)]|v(x)||u(x)| dx\\
	&\leq \frac1{\pi\varsigma_\alpha}\sqrt{\int_{-\pi}^\pi
		\Re[-g(\alpha,x)]|v(x)|^2 dx}\sqrt{\int_{-\pi}^\pi
		\Re[-g(\alpha,x)]|u(x)|^2 dx}\\
	&\le \frac{\sqrt{2}(\kappa_{max}-\kappa_{min})}{\pi\varsigma_\alpha}\int_{-\pi}^\pi
	\Re[-g(\alpha,x)]|u(x)|^2 dx.\\
	&=\frac{\sqrt{2}(\kappa_{max}-\kappa_{min})}{\varsigma_\alpha}{\bf u}^TG{\bf u}.
	\end{align*}
	Thus, the desired result can be obtained just by utilizing the following inequality
	\begin{align*}
	\kappa{\bf u}^TG{\bf u}-\left|{\bf u}^T\left(-{\tilde D}G_\alpha-G_\alpha^T {\tilde D}\right){\bf
		u}\right|&\leq{\bf u}^T(-DG_\alpha-G_\alpha^TD){\bf u}\\
	&\leq\kappa{\bf u}^TG{\bf u}
	+\left|{\bf u}^T\left(-{\tilde D}G_\alpha-G_\alpha^T {\tilde D}\right){\bf u}\right|
	\end{align*}
\end{proof}

\subsection{Toeplitz preconditioner for the discrete one-dimensional fractional diffusion equation}
To solve \eqref{matrixform} is equivalent to recursively solve the following linear systems
\begin{equation}\label{1dlinsys}
{\bf A} {\bf u}^{n}= {b}^n,\quad n=1,2...,N,
\end{equation}
where ${\bf A}={I}_{M}-\eta {D} {G}_{\alpha}$, $\eta=\tau/(2h^{\alpha})$, ${I}_k$ denotes $k\times k$ identity matrix, ${b}^n=({I}_{M}+\eta {D} {G}_{\alpha}){u}^{n-1}+\tau {\bf f}^{n-\frac{1}{2}}$. As explained in the introduction section, a good preconditioner is required for the linear systems in \eqref{1dlinsys}.

For any $m\times m$ diagonal matrix, ${\bf C}=\diag(c_1,c_2,...,c_m)$, denote ${\tt mean}({\bf C})=\frac{1}{m}\sum\limits_{i=1}^{m}c_{i}$.
In this subsection, we propose a Toeplitz preconditioner for the linear systems in \eqref{1dlinsys} such that
\begin{equation}\label{1dtplzprec}
{\bf P}={I}_{M}-\eta\bar{d}{G}_{\alpha},
\end{equation}
where $\bar{d}={\tt mean}({D})$.
In the following, we discuss a computationally effective representation of ${\bf P}^{-1}$, which allows fast matrix-vector multiplication of ${\bf P}^{-1}$.

Let ${\bf v}=(v_1,v_2,...,v_M)^{\rm T}$ and $\tilde{{\bf v}}=(\tilde{v}_1,\tilde{v}_2,...,\tilde{v}_M)^{\rm T}$ be solutions of following linear systems
\begin{align}\label{pregsformlinsys}
{\bf P}{\bf v}={\bf e}_1\equiv(1,0,0,...,0)^{\rm T},\qquad {\bf P}\tilde{{\bf v}}={\bf e}_M\equiv(0,0,...,0,1)^{\rm T}.
\end{align}
According to the Gohberg-Semencul-type formula \cite{gohbergolshevsky}, ${\bf P}^{-1}$ can be expressed as follows
\begin{align}\label{gsformula}
{\bf P}^{-1}=\frac{1}{2v_1}({\bf S}_1{\bf C}_1-{\bf S}_2{\bf C}_2),
\end{align}
where ${\bf S}_1$, ${\bf S}_2$ are skew-circulant matrices with ${\bf v}$, $\bar{{\bf v}}=(-\tilde{v}_M,\tilde{v}_1,...,\tilde{v}_{M-1})^{\rm T}$ as their first columns, respectively; ${\bf C}_1$, ${\bf C}_2$ are circulant matrices with \\$\hat{{\bf v}}=(\tilde{v}_M,\tilde{v}_1,...,\tilde{v}_{M-1})^{\rm T}$, ${\bf v}$ as their first columns, respectively.
From \eqref{pregsformlinsys}, we see that $v_1$ is the first diagonal entry of ${\bf P}^{-1}$. From Lemma \ref{Galpha}, we see that ${\bf P}+{\bf P}^{\rm T}$ is positive definite. Thus,
\begin{align*}
v_1={\bf e}_1^{\rm T}{\bf P}^{-1}{\bf e}_1=\frac{1}{2}{\bf e}_1^{\rm T}({\bf P}^{-1}+{\bf P}^{-\rm T}){\bf e}_1=\frac{1}{2}{\bf e}_1^{\rm T}{\bf P}^{-1}({\bf P}+{\bf P}^{\rm T}){\bf P}^{-\rm T}{\bf e}_1>0,
\end{align*}
which means \eqref{gsformula} is applicable.
Moreover, the Toeplitz linear systems in \eqref{pregsformlinsys} can be efficiently solved by the super fast direct solver proposed in \cite{hoog1987}.

For ${\bf C}\in\mathbb{C}^{m\times n}$, denote by $\Sigma({\bf C})$, the set of singular values of ${\bf C}$. Also denote $\Sigma^2({\bf C})=\{x^2|x\in\Sigma({\bf C})\}$. For any matrix ${\bf C}\in\mathbb{C}^{m\times m}$, denote by $\sigma({\bf C})$, the spectrum of ${\bf C}$. For a number $\lambda$, denote by $\Re(\lambda)$, the real part of $\lambda$.

For any invertible matrix ${\bf C}\in\mathbb{C}^{m\times m}$, define its condition number as $$\cond({\bf C})\triangleq||{\bf C}||_2||{\bf C}^{-1}||_2.$$

\begin{lemma}\label{stabllem}\textnormal{(see \cite[Lemma 2.7]{tianwy2015})}
	For any ${\bf C}\in\mathbb{C}^{m\times m}$, it holds $$\{\Re(\lambda)|\lambda\in\sigma({\bf C})\}\subset\left[\min\limits_{z\in\sigma(({\bf C}+{\bf C}^{*})/2)}z,\max\limits_{z\in\sigma(({\bf C}+{\bf C}^{*})/2)}z\right].$$
\end{lemma}
As a preconditioner, the invertibility is essential.
\begin{proposition}\label{1dinvertibility}
	${\bf P}$ is invertible for any $\alpha\in(1,2)$.
\end{proposition}
\begin{proof}
	By lemme \ref{Galpha}, it is easy to see that ${\bf P}+{\bf P}^{\rm T}$ is positive definite and thus has positive eigenvalues. From Lemma \ref{stabllem}, we see that $\{\Re(\lambda)|\lambda\in\sigma({\bf P})\}\subset(0,+\infty)$. Therefore, ${\bf P}$ is invertible.
\end{proof}

For any Hermitian matrices ${\bf H}_1,{\bf H}_2\in\mathbb{C}^{m\times m}$, denote ${\bf H}_1 \prec {\bf H}_2$ or ${\bf H}_2 \succ {\bf H}_1$ if
${\bf H}_2-{\bf H}_1$ is Hermitian positive definite.
Especially, we denote ${\bf O}\prec{\bf H}_1$ or ${\bf H}_1\succ{\bf O}$, when ${\bf H}_1$ itself is
Hermitian positive definite. Also, we use ${\bf H}_1 \preceq {\bf H}_2$ or ${\bf H}_2 \succeq {\bf H}_1$ to denote a Hermitian positive semidefinite ${\bf H}_2-{\bf H}_1$ and use ${\bf O}\preceq{\bf H}_1$ or ${\bf H}_1\succeq{\bf O}$ to denote a Hermitian positive semidefinite ${\bf H}_1$.

Next, we are to estimate the condition number of the preconditioned matrix ${\bf A}{\bf P}^{-1}$. \begin{proposition}\label{wghtsumbdlem}
	For positive numbers $\xi_i$, $\zeta_i$ ($1\leq i\leq m$),
	it obviously holds that
	\begin{equation*}
	\min\limits_{1\leq i\leq m}\frac{\xi_i}{\zeta_i}\leq\bigg(\sum\limits_{i=1}^{m}\zeta_i\bigg)^{-1}\bigg(\sum\limits_{i=1}^{m}\xi_i\bigg)\leq\max\limits_{1\leq i\leq m}\frac{\xi_i}{\zeta_i}.
	\end{equation*}
\end{proposition}

\begin{theorem}\label{1dcondesti}
	Assume
	\begin{description}
		\item[(i)] for any $x\in(x_{L},x_{R})$, $d(x)\in[\kappa_{\min},\kappa_{\max}]$ for positive constants $\kappa_{\min}$ and $\kappa_{\max}$,
		\item[(ii)] $\kappa_{\max}-\nu_{\alpha}> 0$, with $\nu_{\alpha}=\sqrt{2}(\kappa_{\max}-\kappa_{\min})/\varsigma_{\alpha}$,
		\item[(iii)] $d(x)$ is concave.
	\end{description}
	Then, for any $N\geq 1$, any $M\geq1$, $\Sigma^2({\bf A}{\bf P}^{-1})\subset[\check{s},\hat{s}]$ and thus
	$$\sup\limits_{N,M\geq 1}\cond({\bf A}{\bf P}^{-1})\leq\sqrt{\hat{s}/\check{s}},$$
	where $\check{s}$ and $\hat{s}$ are positive constants independent of $\tau$, $h$ and given by
	\begin{equation*}
	\check{s}=\min\left\{\frac{\kappa_{\max}-\nu_{\alpha}}{\kappa_{\max}},\frac{\kappa_{\min}^2}{\kappa_{\max}^2}\right\},\quad\hat{s}=\max\left\{\frac{\kappa+\nu_{\alpha}}{\kappa_{\min}},\frac{\kappa_{\max}^2}{\kappa_{\min}^2}\right\}.
	\end{equation*}
\end{theorem}
\begin{proof}
	By straightforward calculation,
	\begin{align*}
	{\bf A}^{\rm T}{\bf A}&={I}_{M}-\eta({G}_{\alpha}^{\rm T}{D}+{D}{G}_{\alpha})+\eta^2{G}_{\alpha}^{\rm T}{D}^2{G}_{\alpha},\\
	{\bf P}^{\rm T}{\bf P}&={I}_{M}+\eta\bar{d}{G}+\eta^2\bar{d}^2{G}_{\alpha}^{\rm T}{G}_{\alpha}.
	\end{align*}
	By Theorem \ref{key-estimate}, we see that
	\begin{align}
	{\bf O}&\prec{I}_{M}+(\kappa_{\max}-\nu_{\alpha})\eta{G}+\kappa_{\min}^2\eta^2{G}_{\alpha}^{\rm T}{G}_{\alpha}\notag\\
	&\preceq{\bf A}^{\rm T}{\bf A}\preceq{I}_{M}+(\kappa_{\max}+\nu_{\alpha})\eta{G}+\kappa_{\max}^2\eta^2{G}_{\alpha}^{\rm T}{G}_{\alpha}\label{aatrcontrl}.
	\end{align}
	For any non-zero vector ${\bf y}\in\mathbb{R}^{M\times 1}$, denote ${\bf z}={\bf P}^{-1}{\bf y}$. Then, it holds
	\begin{equation*}
	\frac{{\bf y}^{\rm T}({\bf A}{\bf P}^{-1})^{\rm T}({\bf A}{\bf P}^{-1}){\bf y}}{{\bf y}^{\rm T}{\bf y}}=\frac{{\bf z}^{\rm T}{\bf A}^{\rm T}{\bf A}{\bf z}}{{\bf z}^{\rm T}{\bf P}^{\rm T}{\bf P}{\bf z}}.
	\end{equation*}
	By \eqref{aatrcontrl},
	\begin{align}\label{nrangeesti}
	&\frac{{\bf z}^{\rm T}[{I}_{M}+(\kappa_{\max}-\nu_{\alpha})\eta{G}+\kappa_{\min}^2\eta^2{G}_{\alpha}^{\rm T}{G}_{\alpha}]{\bf z}}{{\bf z}^{\rm T}[{I}_{M}+\eta\bar{d}{G}+\eta^2\bar{d}^2{G}_{\alpha}^{\rm T}{G}_{\alpha}]{\bf z}}\notag\\
	&\leq\frac{{\bf z}^{\rm T}{\bf A}^{\rm T}{\bf A}{\bf z}}{{\bf z}^{\rm T}{\bf P}^{\rm T}{\bf P}{\bf z}}\leq\frac{{\bf z}^{\rm T}[{I}_{M}+(\kappa_{\max}+\nu_{\alpha})\eta{G}+\kappa_{\max}^2\eta^2{G}_{\alpha}^{\rm T}{G}_{\alpha}]{\bf z}}{{\bf z}^{\rm T}[{I}_{M}+\eta\bar{d}{G}+\eta^2\bar{d}^2{G}_{\alpha}^{\rm T}{G}_{\alpha}]{\bf z}}.
	\end{align}
	By Proposition \eqref{wghtsumbdlem} and \eqref{nrangeesti},
	\begin{equation*}
	\check{s}=\min\left\{1,\frac{\kappa_{\max}-\nu_{\alpha}}{\kappa_{\max}},\frac{\kappa_{\min}^2}{\kappa_{\max}^2}\right\}\leq\frac{{\bf z}^{\rm T}{\bf A}^{\rm T}{\bf A}{\bf z}}{{\bf z}^{\rm T}{\bf P}^{\rm T}{\bf P}{\bf z}}\leq\max\left\{1,\frac{\kappa+\nu_{\alpha}}{\kappa_{\min}},\frac{\kappa_{\max}^2}{\kappa_{\min}^2}\right\}=\hat{s}.
	\end{equation*}
	During the proof above, there is no constraint on $M$ and $N$. Thus, for any $N\geq 1$, any $M\geq1$, $\Sigma^2({\bf A}{\bf P}^{-1})\subset[\check{s},\hat{s}]$ and $\sup\limits_{N,M\geq 1}\cond({\bf A}{\bf P}^{-1})\leq\sqrt{\hat{s}/\check{s}}$.
\end{proof}

Similar to proof of Theorem \ref{1dcondesti}, one can prove following theorem.
\begin{theorem}\label{1dcondesti1}
	Assume
	\begin{description}
		\item[(i)] for any $x\in(x_{L},x_{R})$, $d(x)\in[\kappa_{\min},\kappa_{\max}]$ for positive constants $\kappa_{\min}$ and $\kappa_{\max}$,
		\item[(ii)] $\kappa_{\min}-\nu_{\alpha}> 0$, with $\nu_{\alpha}=\sqrt{2}(\kappa_{\max}-\kappa_{\min})/\varsigma_{\alpha}$,
		\item[(iii)] $d(x)$ is convex.
	\end{description}
	Then, for any $N\geq 1$, any $M\geq1$, $\Sigma^2({\bf A}{\bf P}^{-1})\subset[\check{s},\hat{s}]$ and thus
	$$\sup\limits_{N,M\geq 1}\cond({\bf A}{\bf P}^{-1})\leq\sqrt{\hat{s}/\check{s}},$$
	where $\check{s}$ and $\hat{s}$ are positive constants independent of $\tau$, $h$ and given by
	\begin{equation*}
	\check{s}=\min\left\{\frac{\kappa_{\min}-\nu_{\alpha}}{\kappa_{\max}},\frac{\kappa_{\min}^2}{\kappa_{\max}^2}\right\},\quad\hat{s}=\max\left\{\frac{\kappa_{\min}+\nu_{\alpha}}{\kappa_{\min}},\frac{\kappa_{\max}^2}{\kappa_{\min}^2}\right\}.
	\end{equation*}
\end{theorem}
\begin{remark}
	Theorems \ref{1dcondesti}--\ref{1dcondesti1} show that $\cond({\bf A}{\bf P}^{-1})$ has an upper bound independent of $\tau$ and $h$ under certain assumptions on the coefficient function $d$. Thus, Krylov subspace method for such preconditioned linear systems converges linearly and independently on the discretization step-sizes.
\end{remark}

\section{Extension to Two-dimensional OSFDE}\label{secondorder2D}
In this section, we study the following two-dimensional OSFDE \cite{tadjeran2007813}:
\begin{equation}\label{fde}
\begin{array}{l}
\displaystyle{\frac{\partial u(x,y,t)}{\partial t}=
	d(x,y)\,{_{x_L}}D_x^{\alpha}u(x,y,t)+e(x,y){_{y_L}}D_{y}^{\beta}u(x,y,t)+f(x,y,t)}, \\
~~~~~~~~~~~~~~~~~~~~~~~~~~~~~~~~~~~~~~~~~~~~~~~~~~~~~~~~~~~~~~~~~~~~~~~~~~~~~~~(x,y)\in\Omega,~t\in(0,T],\\
\vspace{2mm}
u(x,y,t)=0, \qquad (x,y)\in\partial\Omega,~ t\in[0,T],\\
u(x,y,0)=\varphi(x,y),\qquad (x,y)\in\bar\Omega,
\end{array}
\end{equation}
where $\alpha,\beta\in(1,2)$, $\Omega=(x_L,x_R)\times(y_L,y_R)$ and $d(x,y)$, $e(x,y)$ are nonnegative functions, ${_{x_L}}D_x^{\alpha}u(x,y,t)$ denotes the $\alpha$-order RL derivative with respect to $x$ direction defined as
\begin{equation*}
\begin{array}{l}\vspace{2mm}
\displaystyle{\,{_{x_L}}D_x^{\alpha}u(x,y,t)=\dfrac{1}{\Gamma(2-\alpha)}\dfrac{\partial^2}{\partial
		x^2}\int_{x_L}^x \dfrac{u(\xi,y,t)}{(x-\xi)^{\alpha-1}}\,d\xi},
\end{array}
\end{equation*}
${_{y_L}}D_y^{\beta}u(x,y,t)$ can be defined in a similar way.

To state a finite difference scheme for \eqref{fde}, we need more notations.
Let $\tau=T/N$, $h_1=\frac{x_R-x_L}{M_1+1}$, $h_2=\frac{y_R-y_L}{M_2+1}$,
where $M_1$, $M_2$ and $N$ are some positive integers.
For $i=0,1,\ldots,M_1+1$, $j=0,1,\ldots,M_2+1$ and $n = 0,1,\ldots,N$, denote $x_i=ih_1$, $y_j=jh_2$ and $t_n = n\tau$.
Denote $t_{n-\frac12}=\frac{t_n+t_{n-1}}{2}$ for $n=1,2,...,N$. Let $\bar{\Omega}_h=\{(x_i,y_j)|0\leq i\leq M_1+1, 0\leq j\leq M_2+1\}$, ${\Omega}_h={\bar{\Omega}_h}\cap{\Omega}$, $\partial{{\Omega}_h}={\bar{\Omega}_h}\cap{\partial{\Omega}}$.
Furthermore, denote $d_{i,j}=d(x_i,y_j)$, $e_{i,j}=e(x_i,y_j)$, $f_{i,j}^{n-\frac12}=f(x_i,y_j,t_{n-\frac12})$, and $\varphi_{i,j}=\varphi(x_i,y_j)$, and let $u_{i,j}^n$ be the numerical approximation of $u(x_i,y_j,t_n)$. Then, in a similar way with the one-dimensional case, we can derive the Crank-Nicolson scheme for the two-dimensional problem \eqref{fde} as the following
\begin{align}
\frac{u_{i,j}^{n}-u_{i,j}^{n-1}}{\tau}=&
\frac1{2h^{\alpha}}d_{i,j}\sum_{k=0}^{i}w_{k}^{(\alpha)}\left(u_{i-k+1,j}^{n-1}
+u_{i-k+1,j}^{n}\right)\notag\\
&+\frac1{2h^{\beta}}e_{i,j}\sum_{k=0}^{j}w_{k}^{(\beta)}\left(u_{i,j-k+1}^{n-1}
+u_{i,j-k+1}^{n}\right)+ f_{i,j}^{n-\frac{1}{2}}+{\hat R}_{i,j}^{n-\frac12},\notag\\
&1\leq i\leq M_1,~1\leq j\leq M_2,~ 1\leq n\leq N,\label{2d-scheme}
\end{align}
where ${\hat R}_{i,j}^{n-\frac12}\leq c_3(\tau^2+h_1^2+h_2^2)$ for a positive constant $c_3$.

Take
\begin{align*}
&u^{n}=[u^{n}_{1,1},u^{n}_{2,1},\ldots,u^{n}_{M_1,1},u^{n}_{1,2},\ldots,u^{n}_{M_1,2},
\ldots,u^{n}_{1,M_2},\ldots,u^{n}_{M_1,M_2}]^T,\\
&f^{n-\frac12}=[f^{n-\frac12}_{1,1},f^{n-\frac12}_{2,1},\ldots,f^{n-\frac12}_{M_1,1},f^{n-\frac12}_{1,2},
\ldots,f^{n-\frac12}_{M_1,2},
\ldots,f^{n-\frac12}_{1,M_2},\ldots,f^{n-\frac12}_{M_1,M_2}]^T,\\
&D=\diag(d_{1,1},d_{2,1},\ldots,d_{M_1,1},d_{1,2},\ldots,d_{M_1,2},\ldots,d_{M_1,M_2}),\\
&E=\diag(e_{1,1},e_{2,1},\ldots,e_{M_1,1},e_{1,2},\ldots,e_{M_1,2},\ldots,e_{M_1,M_2}).
\end{align*}
Omitting the small term ${\hat R}_{i,j}^{n-\frac12}$ in \eqref{2d-scheme}, the finite difference scheme in matrix form for \eqref{fde} can be given as:
\begin{align}\label{2Dmatrixform}
&\frac1{\tau}\left(u^{n}-u^{n-1}\right)=\left(\frac{1}{2h_1^\alpha}D(I\otimes G_\alpha)+\frac{1}{2h_2^\beta}E(G_\beta\otimes I)\right)\left(u^{n-1}+u^{n} \right)+ f^{n-\frac{1}{2}},\notag\\
&1\leq n\leq N,
\end{align}
where $I$ is the identity matrix, the symbol `$\otimes$' denotes the Kronecker product, and $G_\beta$ has the similar definition to $G_\alpha$.

\subsection{Stability and Convergence of the Two-dimensional Problem}
To discuss the stability and convergence of scheme \eqref{2Dmatrixform},
we denote
$$A=\frac{1}{2h_1^\alpha}D(I\otimes G_\alpha)+\frac{1}{2h_2^\beta}E(G_\beta\otimes I),$$
and introduce a set:
\begin{align*}
{\mathfrak{D}}=\{&X|X\succ{\bf O},~ -{\cal H}(XA)\succeq{\bf O},~{\rm cond}(X)\leq c~{\rm for~}c{\rm ~independent~of~}\tau,~h_1~\\
& {\rm and}~h_2\}.
\end{align*}

Now we present the stability of the scheme \eqref{2Dmatrixform}.
\begin{theorem}\label{stable-2D-1}
	For any $Q\in{\mathfrak{D}}$, the finite difference scheme \eqref{2Dmatrixform} is unconditionally stable and its solution satisfies the following estimate
	$$\left\|u^{n}\right\|_{Q}^2\leq \exp(2T)\left\|\varphi\right\|_{Q}^2+[\exp(2T)-1]\max_{1\leq k\leq n}\left\|f^{k-\frac12}\right\|_{Q}^2,\quad n=1,2,\ldots,N,$$
	where $\|\cdot\|_{Q}$ is defined as $\|v\|_{Q}^2:=h v^TQv$.
\end{theorem}
\begin{proof}
	Multiplying $h\left(u^{n-1}+u^{n} \right)^TQ$ on the both sides of \eqref{2Dmatrixform}, we get
	\begin{align*}
	&\frac1{\tau}h\left(u^{n-1}+u^{n} \right)^TQ\left(u^{n}-u^{n-1}\right)\\
	&=h\left(u^{n-1}+u^{n} \right)^T{QA}\left(u^{n-1}+u^{n} \right)+ h\left(u^{n-1}+u^{n} \right)^TQf^{n-\frac{1}{2}}.
	\end{align*}
	Since ${\cal H}(QA)$ is negative semi-definite, we have
	$$h\left(u^{n-1}+u^{n} \right)^T{QA}\left(u^{n-1}+u^{n} \right)=h\left(u^{n-1}+u^{n} \right)^T{\cal H}({QA})\left(u^{n-1}+u^{n} \right)\leq 0.$$
	Then it follows
	$$ h(u^{n})^TQu^{n}- h(u^{n-1})^TQu^{n-1}\le\tau h(u^{n})^TQf^{n-\frac{1}{2}}
	+\tau h(u^{n-1})^TQf^{n-\frac{1}{2}}.$$
	The rest of the proof is similar to that in Theorem \ref{stable-1D}.
\end{proof}

With Theorem \ref{stable-2D-1}, the convergence of scheme \eqref{2Dmatrixform} can be directly obtained:
\begin{theorem}\label{convergence2D-1}
	Let $u(x_i,y_j,t_n)$ be the exact solution of \eqref{fde} and smooth
	enough, $u_{i,j}^n$ be the solution of finite difference scheme \eqref{2Dmatrixform}. Denote $e_{i,j}^n=u(x_i,y_j,t_n)-u_{i,j}^n$, $0\leq i\leq M_1+1$, $0\leq j\leq M_2+1$, $0\leq n\leq N$. For any $Q\in{\mathfrak{D}}$, there exists a positive constant $c_4$ such that
	$$\|e^n\|\leq c_4(\tau^2+h_1^2+h_2^2).$$
\end{theorem}

The remaining and important thing is to give the feature of the set ${\mathfrak{D}}$. However, it seems difficult to depict all the elements of ${\mathfrak{D}}$. In the following Corollaries \ref{case1lemma} and \ref{case2lemma}, we show that there are some matrices belong to ${\mathfrak{D}}$ when the variable coefficients $d(x,y),e(x,y)$ satisfy some certain conditions, this ensures that ${\mathfrak{D}}$ is not an empty set which is necessary for the stability and convergence. We discuss the existence of those matrices in two cases:
\\
$\bullet$ {\bf Case 1} When $d(x,y),e(x,y)$ are separable respect to $x$ and $y$.

In this case, we denote $d(x,y)={\tilde d}(x){\hat d}(y)$ and $e(x,y)={\tilde e}(x){\hat e}(y)$, and take
\begin{align*}
&{\tilde D}=\diag({\tilde d}_1,{\tilde d}_2,\ldots,{\tilde d}_{M_1}),~~{\hat D}=({\hat d}_1,{\hat d}_2,\ldots,{\hat d}_{M_2}),\\
&{\tilde E}=\diag({\tilde e}_1,{\tilde e}_2,\ldots,{\tilde e}_{M_1}),~~{\hat E}=({\hat e}_1,{\hat e}_2,\ldots,{\hat e}_{M_2}).
\end{align*}
Then $D={\hat D}\otimes{\tilde D}$, $E={\hat E}\otimes{\tilde E}$.
\begin{corollary}\label{case1lemma}
	If ${\tilde d}_-\leq{\tilde d}(x)\leq {\tilde d}_+$ and ${\hat e}_-\leq{\hat e}(y)\leq {\hat e}_+$ for some positive constants ${\tilde d}_-,{\tilde d}_+,{\hat e}_-$ and ${\hat e}_+$, then ${\hat E}^{-1}\otimes{\tilde D}^{-1}\in{\mathfrak{D}}$.
\end{corollary}
\begin{proof}
	We have $A=\frac{1}{2h_1^\alpha}({\hat D}\otimes {\tilde D}G_\alpha)+\frac{1}{2h_2^\beta}({\hat E}G_\beta\otimes {\tilde E})$, then
	$${\cal H}\left(({\hat E}^{-1}\otimes{\tilde D}^{-1})A\right)=\frac{1}{4h_1^\alpha}\left({\hat E}^{-1}{\hat D}\otimes (G_\alpha+G_\alpha^T)\right)+\frac{1}{4h_2^\beta}\left((G_\beta+G_\beta^T)\otimes {\tilde D}^{-1}{\tilde E}\right),$$
	which is negative semi-definite. Thus ${\hat E}^{-1}\otimes{\tilde D}^{-1}\in{\mathfrak{D}}$.
\end{proof}
$\bullet$ {\bf Case 2} {When $d(x,y)$ and $e(x,y)$ are non-separable.}

As in Lemma \ref{generatingfunction_rario}, we denote
$$\varsigma_\beta\triangleq\min_{x}\frac{\Re[-g(\beta,x)]}{|g(\beta,x)|}=\left|\cos\left(\frac{\beta}{2}\pi\right)\right|,$$
where $g(\beta,x)$ is the generating function of matrix $G_\beta$.

%Take $A=\frac{1}{2h^\alpha}D(I\otimes G_\alpha)+\frac{1}{2h^\beta}E(G_\beta\otimes I)$. Then we show that the symmetric part of matrix $A'\triangleq D'A$ is negative semidefinite when $d(x,y)$ and $e(x,y)$ satisfy some conditions, where $D'$ can be $I$, $D^{-1}$ and $E^{-1}$, \textcolor{red}{\bf why just consider $I$, $D^{-1}$, $E^{-1}$ these three matrices? It could be any matrix $D^{\prime}$ such that $D^{\prime}D(I\otimes G_{\alpha})$ and $D^{\prime}E(G_{\beta}\otimes I)$ are both negative definite simultaneously}
\begin{corollary}\label{case2lemma}\mbox{}\par
	\begin{description}
		\item[(i)] Assume that
		$$0\leq\kappa_{min}^d(y)\leq d(x,y)\leq\kappa_{max}^d(y)<\infty~\mbox{for every $(x,y)$},$$ $$0\leq\kappa_{min}^e(x)\leq e(x,y)\leq\kappa_{max}^e(x)<\infty~\mbox{for every $(x,y)$}.$$
		Then, $I\in{\mathfrak{D}}$ if the following conditions are fulfilled
		\begin{align}\label{cond2-1}
		&\kappa^d(y)-\frac{\sqrt{2}\big(\kappa_{max}^d(y)-\kappa_{min}^d(y)\big)}{\varsigma_\alpha}\geq0~\mbox{for every $y$},\\\label{cond2-11}
		&\kappa^e(x)-\frac{\sqrt{2}\big(\kappa_{max}^e(x)-\kappa_{min}^e(x)\big)}{\varsigma_\beta}\geq0~\mbox{for every $x$},
		\end{align}
		where $\kappa^d(y)=\kappa_{max}^d(y)$ when $d(x,y)$ is a concave function of $x$, $\kappa^d(y)=\kappa_{min}^d(y)$ when $d(x,y)$ is a convex function of $x$, $\kappa^e(x)=\kappa_{max}^e(x)$ when $e(x,y)$ is a concave function of $y$, $\kappa^e(x)=\kappa_{min}^e(x)$ when $e(x,y)$ is a convex function of $y$;
		\item[(ii)]  Assume that
		$$0\leq\kappa_{min}^{e'}(x)\leq \frac{e(x,y)}{d(x,y)}\leq\kappa_{max}^{e'}(x)<\infty \quad \mbox{with}\quad  0<d(x,y)<\infty.$$
		Then $D^{-1}\in{\mathfrak{D}}$ if the following condition is fulfilled
		\begin{equation}\nonumber \kappa^{e'}(x)-\frac{\sqrt{2}\big(\kappa_{max}^{e'}(x)-\kappa_{min}^{e'}(x)\big)}{\varsigma_\beta}\geq0,
		\end{equation}
		where $\kappa^{e'}(x)=\kappa_{max}^{e'}(x)$ when $\frac{e(x,y)}{d(x,y)}$ is a concave function of $y$, and $\kappa^{e'}(x)=\kappa_{min}^{e'}(x)$ when $\frac{e(x,y)}{d(x,y)}$ is a convex function of $y$;
		\item[(iii)] Assume that
		$$0\leq\kappa_{min}^{d'}(y)\leq \frac{d(x,y)}{e(x,y)}\leq\kappa_{max}^{d'}(y)<\infty \quad \mbox{with}\quad  0<e(x,y)<\infty.$$
		Then $E^{-1}\in{\mathfrak{D}}$ if the following condition is fulfilled
		\begin{equation}\nonumber \kappa^{d'}(y)-\frac{\sqrt{2}\big(\kappa_{max}^{d'}(y)-\kappa_{min}^{d'}(y)\big)}{\varsigma_\alpha}\geq0,
		\end{equation}
		where $\kappa^{d'}(y)=\kappa_{max}^{d'}(y)$ when $\frac{d(x,y)}{e(x,y)}$ is a concave function of $x$, and $\kappa^{d'}(y)=\kappa_{min}^{d'}(y)$ when $\frac{d(x,y)}{e(x,y)}$  is a convex function of $x$.
	\end{description}
\end{corollary}

\begin{proof}
	We firstly prove ${\bf (i)}$. Denote
	$$K^d=\diag\big(k^d(y_1),k^d(y_2),\ldots,k^d(y_{M_2})\big),\quad
	K^e=\diag\big(k^e(x_1),k^e(x_2),\ldots,k^e(x_{M_1})\big).$$
	Take ${\tilde D}=D-K^d\otimes I$ and ${\tilde E}=E-I\otimes K^e$.
	Then
	$$A=\frac1{2h^\alpha}\left[(K^d\otimes G_\alpha)+{\tilde D}(I\otimes G_\alpha)\right]+\frac1{2h^\beta}\left[(G_\beta\otimes K^e)+{\tilde E}(G_\beta\otimes I)\right].$$
	For any ${\bf u}=[u_{1,1},u_{2,1},\ldots,u_{M_1,1},u_{1,2},\ldots,u_{M_1,2},
	\ldots,u_{1,M_2},\ldots,u_{M_1,M_2}]^T$, we have
	\begin{align*}
	2{\bf u}^T{\cal H}(A){\bf u}
	=&\frac1{2h^\alpha}\left[{\bf u}^T\big(K^d\otimes (G_\alpha+G_\alpha^T)\big){\bf u}+{\bf u}^T\left({\tilde D}(I\otimes G_\alpha)+(I\otimes G_\alpha^T){\tilde D}\right){\bf u}\right]\\
	&+\frac1{2h^\beta}\left[{\bf u}^T\big((G_\beta+G_\beta^T)\otimes K^e\big){\bf u}+{\bf u}^T\left({\tilde E}(G_\beta\otimes I)+(G_\beta^T\otimes I){\tilde E}\right){\bf u}\right].
	\end{align*}
	Referring to the proof of Theorem \ref{key-estimate}, it is easy to obtain
	\begin{align*}
	\left|{\bf u}^T\left({\tilde D}(I\otimes G_\alpha)+(I\otimes G_\alpha^T){\tilde D}\right){\bf u}\right|\le& \frac{-\sqrt{2}}{\varsigma_\alpha}{\bf u}^T\left(K_\alpha\otimes (G_\alpha+G_\alpha^T)\right){\bf u},\\
	\left|{\bf u}^T\left({\tilde E}(G_\beta\otimes I)+(G_\beta^T\otimes I){\tilde E}\right){\bf u}\right|\le& \frac{-\sqrt{2}}{\varsigma_\beta}{\bf u}^T\left((G_\beta+G_\beta^T)\otimes K_\beta\right){\bf u},
	\end{align*}
	where
	$$K_\alpha=\diag\Big(\kappa_{max}^d(y_1)-\kappa_{min}^d(y_1),
	\kappa_{max}^d(y_2)-\kappa_{min}^d(y_2),\ldots,\kappa_{max}^d(y_{M_2})-\kappa_{min}^d(y_{M_2}) \Big),$$
	$$K_\beta=\diag\Big(\kappa_{max}^e(x_1)-\kappa_{min}^e(x_1),
	\kappa_{max}^e(x_2)-\kappa_{min}^e(x_2),\ldots,\kappa_{max}^e(x_{M_1})-\kappa_{min}^e(x_{M_1}) \Big).$$
	So
	\begin{align*}
	-2{\bf u}^T{\cal H}(A){\bf u}\geq&\frac1{2h^\alpha}{\bf u}^T\left(\Big(K^d-\frac{\sqrt{2}}{\varsigma_\alpha}K_\alpha\Big)\otimes (-G_\alpha-G_\alpha^T)\right){\bf u}\\
	&+\frac1{2h^\beta}
	{\bf u}^T\left((-G_\beta-G_\beta^T)\otimes \Big(K^e-\frac{\sqrt{2}}{\varsigma_\beta}K_\beta\Big)\right){\bf u}.
	\end{align*}
	Which implies that ${\cal H}(A)$ is negative semi-definite if the conditions in \eqref{cond2-1}-\eqref{cond2-11} hold. Hence $I\in{\mathfrak{D}}$.
	
	Similarly, one can show ${\bf (ii)}$ and ${\bf (iii)}$.
\end{proof}

\subsection{The Two-Dimensional Toeplitz Preconditioner}
In this subsection, we extend the Toeplitz preconditioner to two-dimensional case. To solve \eqref{2Dmatrixform} is equivalent to solve the following $N$ linear systems
\begin{equation}\label{2dlinsys}
{\bf A}{u}^n={b}^n,\quad n=1,2,...,N,
\end{equation}
where $I_k$ denotes the $k\times k$ identity, ${\bf A}={I}_{\hat{M}}+{D}{B}_x+{E}{B}_y$, ${B}_x=-\eta_x({I}_{M_2}\otimes{G}_{\alpha})$, ${B}_y=-\eta_y({G}_{\beta}\otimes{I}_{M_1})$ $\hat{M}=M_1M_2$, $\eta_x=\tau/(2h_1^{\alpha})$, $\eta_y=\tau/(2h_2^{\beta})$, ${\bf b}^{n}=({I}_{\hat{M}}-{D}{B}_x-{E}{B}_y){u}^{n-1}+\tau{\bf f}^{n-\frac{1}{2}}$.
Our two-level Toeplitz preconditioner for preconditioning  \eqref{2dlinsys} is defined as follows
\begin{equation}\label{2dpreconditioner}
{\bf P}={I}_{\hat{M}}+\bar{d}{B}_x+\bar{e}{B}_{y},
\end{equation}
where $\bar{d}={\tt mean}({D})$, $\bar{e}={\tt mean}({E})$.
The preconditioned Krylov subspace method with preconditioner ${\bf P}$ is employed to solve the linear systems in \eqref{2dlinsys}. Hence, in each iteration, it requires to compute some matrix-vector multiplications like ${\bf P}^{-1}{\bf z}$ for some randomly given ${\bf z}$, i.e., it requires to solve the linear system of the form
\begin{equation}\label{implmtsys}
{\bf P}{\bf x}={\bf z}.
\end{equation}
Next, we introduce a multigrid method to solve \eqref{implmtsys}.

For the choices of coarse-gird matrices, interpolation and restriction, we refer to the geometric grid coarsening, piecewise linear interpolation and its transpose. For the choice of pre-smoothing iteration, we refer to the block Jacobi iteration, i.e.,
\begin{equation}\label{presmoother}
{\bf x}^{k+1}={\bf x}^{k}+{\bf T}_x^{-1}({\bf z}-{\bf P}{\bf x}^{k}),
\end{equation}
where ${\bf T}_x={I}_{\hat{M}}+\bar{d}{B}_x$ is the block diagonal part of ${\bf P}$, ${\bf x}^{k}$ is an initial guess of ${\bf x}$ in \eqref{implmtsys}. Since ${\bf T}_x$ is a block diagonal matrix with identical Toeplitz blocks, its inversion, ${\bf T}_x^{-1}$ can be computed efficiently with the help of Gohberg-Semencul-type formula as discussed in Section 2. For the choice of post-smoother, we refer to the block Jacobi iteration for the permuted linear system, i.e,
\begin{equation}\label{postsmoother}
{\bf x}^{k+1}={\bf x}^{k}+{\bf T}_y^{-1}({\bf z}-{\bf P}{\bf x}^{k}),
\end{equation}
where ${\bf T}_y={I}_{\hat{M}}+\bar{e}{B}_y$, ${\bf x}^{k}$ is an initial guess of ${\bf x}$ in \eqref{implmtsys}. One can easily find a $x$-$y$ ordering permutation  matrix ${\bf Q}$ such that
\begin{equation}\label{defpermut}
{\bf T}_y={\bf Q}^{\rm T}({I}_{\hat{M}}-\bar{e}\eta_y{I}_{M_1}\otimes{G}_{\beta}){\bf Q}.
\end{equation}
Thus, ${\bf T}_y^{-1}={\bf Q}^{\rm T}({I}_{\hat{M}}-\bar{e}\eta_y{I}_{M_1}\otimes{G}_{\beta})^{-1}{\bf Q}$, which means the implementation of \eqref{postsmoother} still requires to compute an inversion of a block diagonal matrix with identical Toeplitz blocks. Therefore, \eqref{postsmoother} can still be fast implemented using the Gohberg-Semencul-type formula.
Similar to proof of Proposition \ref{1dinvertibility}, one can prove the following proposition.
\begin{proposition}\label{2dinvertibility}
	${\bf P}$ defined in \eqref{2dpreconditioner} is invertible for any $\alpha\in(1,2)$.
\end{proposition}
%\begin{theorem}\label{2dsprblclstthm}
%
%\end{theorem}

\begin{theorem}\label{2dclstthm}
	Let $d(x,y)\equiv\nu_1a(x,y)$ and $e(x,y)\equiv\nu_2a(x,y)$ for any $(x,y)\in\Omega$ with nonnegative constants $\nu_1$ and $\nu_2$. Assume
	\begin{description}
		\item[(i)] $a(x,y)\in[\check{a},\hat{a}]$ with $\check{a}>0$ for any $(x,y)\in\Omega$,
		\item[(ii)]For any $x\in(x_{L},x_{R})$, $a(x,\cdot)$ is convex or concave on $y\in(y_{L},y_{R})$; and for any $y\in(y_{L},y_{R})$, $a(\cdot,y)$ is convex or concave on $x\in(x_L,x_R)$.
		\item[(iii)] $\check{c}_{1}=\inf\limits_{y\in(y_L,y_R)}[\tilde{\mathcal{M}}_1(y)-\sqrt{2}(\hat{\mathcal{M}}_1(y)-\check{\mathcal{M}}_1(y))/\varsigma_{\alpha}]>0$ with $\hat{\mathcal{M}}_1(y)=:\sup\limits_{x\in(x_L,x_R)}a(x,y)$ and $\check{\mathcal{M}}_1(y)=:\inf\limits_{x\in(x_L,x_R)}a(x,y)$,
		$$
		\tilde{\mathcal{M}}_1(y)=\begin{cases}
		\check{\mathcal{M}}_1(y),\quad {\rm if}~ a(\cdot,y){\rm~is~convex},\\
		\hat{\mathcal{M}}_1(y),\quad{\rm if}~ a(\cdot,y){\rm~is~concave},
		\end{cases}
		$$
		$\check{c}_{2}=\inf\limits_{x\in(x_L,x_R)}[\tilde{\mathcal{M}}_2(x)-\sqrt{2}(\hat{\mathcal{M}}_2(x)-\check{\mathcal{M}}_2(x))/\varsigma_{\beta}]>0$ with $\hat{\mathcal{M}}_2(x)=:\sup\limits_{y\in(y_L,y_R)}a(x,y)$ and $\check{\mathcal{M}}_2(x)=:\inf\limits_{y\in(y_L,y_R)}a(x,y)$,
		$$
		\tilde{\mathcal{M}}_2(x)=\begin{cases}
		\check{\mathcal{M}}_2(x),\quad {\rm if}~ a(x,\cdot){\rm~is~convex},\\
		\hat{\mathcal{M}}_2(x),\quad{\rm if}~ a(x,\cdot){\rm~is~concave}.
		\end{cases}
		$$
	\end{description}
	Then, for any positive integers, $N$, $M_1$ and $M_2$, it holds $\Sigma^2({\bf A}{\bf P}^{-1})\subset[\check{s},\hat{s}]$ and thus $$\sup\limits_{M_1,M_2,N\geq 1}\cond({\bf A}{\bf P}^{-1})\leq\sqrt{\hat{s}/\check{s}},$$
	where $\check{s}$, $\hat{s}$ are positive constants independent of $\tau$, $h_1$ and $h_2$:
	\begin{align*}
	&\check{s}=\min\left\{\frac{\check{c}_1}{\hat{a}},\frac{\check{c}_2}{\hat{a}},\frac{\check{a}^2}{\hat{a}^2}\right\},~\hat{s}=\max\left\{\frac{\hat{c}_1}{\check{a}},\frac{\hat{c}_2}{\check{a}},\frac{\hat{a}^2}{\check{a}^2}\right\},\\
	&\hat{c}_1=\sup\limits_{y\in(y_L,y_R)}\Big[\tilde{\mathcal{M}}_1(y)+\frac{\sqrt{2}}{\varsigma_{\alpha}}(\hat{\mathcal{M}}_1(y)-\check{\mathcal{M}}_1(y))\Big],\\
	&\hat{c}_2=\sup\limits_{x\in(x_L,x_R)}\Big[\tilde{\mathcal{M}}_2(x)+\frac{\sqrt{2}}{\varsigma_{\beta}}(\hat{\mathcal{M}}_2(x)-\check{\mathcal{M}}_2(x))\Big].
	\end{align*}
\end{theorem}
\begin{proof}
	Denote $${ D}_a=\diag(a_{1,1},a_{2,1},...,a_{M_1,1},a_{1,2},a_{2,2},...,a_{M_1,2},......,a_{1,M_2},a_{2,M_2},...,a_{M_1,M_2})$$ with $a_{i,j}=a(x_{i},y_{j})$. Also, denote $\bar{a}={\tt mean}({ D}_a)$.
	By straightforward calculation,
	\begin{align}
	{\bf A}^{\rm T}{\bf A}&={ I}_{\hat{M}}+\nu_1({ B}_x^{\rm T}{ D}_a+{ D}_a{ B}_x)+\nu_2({ B}_y^{\rm T}{ D}_a+{ D}_a{ B}_y)+{ W}^{\rm T}{ D}_a^{2}{ W},\label{atra}\\
	{\bf P}^{\rm T}{\bf P}&={ I}_{\hat{M}}+\nu_1\bar{a}({ B}_x^{\rm T}+{ B}_x)+\nu_2\bar{a}({ B}_y^{\rm T}+{ B}_y)+\bar{a}^2{ W}^{\rm T}{ W},\label{ptrp}
	\end{align}
	where ${ W}=\nu_1{ B}_x+\nu_2{ B}_y$. Rewrite ${ D}_{a}=\diag({ D}_{a,1},{ D}_{a,2},...,{ D}_{a,M_2})$ with \\${ D}_{a,i}=\diag(a_{1,i},a_{2,i},...,a_{M_1,i})$. Then, it is easy to see that
	${ B}_x^{\rm T}{ D}_a+{ D}_a{ B}_x=\diag({ H}_1,{ H}_2,...,{ H}_{M_2})$ with ${ H}_i=-\eta_x({ D}_{a,i}{ G}_{\alpha}+{ G}_{\alpha}^{\rm T}{ D}_{a,i})$. Denote $l_1(y)=\tilde{\mathcal{M}}_1(y)-\sqrt{2}(\hat{\mathcal{M}}_1(y)-\check{\mathcal{M}}_1(y))/\varsigma_{\alpha}$ and $s_1(y)=\tilde{\mathcal{M}}_1(y)+\sqrt{2}(\hat{\mathcal{M}}_1(y)-\check{\mathcal{M}}_1(y))/\varsigma_{\alpha}$  . Then, applying Theorem \ref{key-estimate} to ${\bf (i)}$--${\bf (iii)}$, we have
	\begin{equation*}
	-\check{c}_1\eta_x(G_{\alpha}+G_{\alpha}^{\rm T})\preceq l_1(y_i)\eta_x{ G}\preceq{ H}_i\preceq s_1(y_i)\eta_x{ G}\preceq-\hat{c}_1\eta_x(G_{\alpha}+G_{\alpha}^{\rm T}).
	\end{equation*}
	Therefore,
	\begin{equation}\label{xpartest}
	{\bf  O}\prec\check{c}_1({ B}_x^{\rm T}+{ B}_x)\preceq{ B}_x^{\rm T}{ D}_a+{ D}_a{ B}_x\preceq\hat{c}_1({ B}_x^{\rm T}+{ B}_x),
	\end{equation}
	where the first `$\prec$' is obvious.
	Recall the permutation matrix defined in \eqref{defpermut}. Denote $\tilde{ B}_y:={\bf Q}{ B}_y{\bf Q}^{\rm T}=-\eta_y{ I}_{M_1}\otimes{ G}_{\beta}$, $\tilde{ D}_a=\diag(\tilde{ D}_{a,1},\tilde{ D}_{a,2},...,\tilde{ D}_{a,M_1})$ with $\tilde{ D}_{a,i}=\diag(a_{i,1},a_{i,2},...,a_{i,M_2})$. Then, it is easy to check that
	\begin{align*}
	{ B}_y^{\rm T}{ D}_a+{ D}_a{ B}_y={\bf  Q}^{\rm T}(\tilde{ B}_y^{\rm T}\tilde{ D}_a+\tilde{ D}_a\tilde{ B}_y){\bf  Q}.
	\end{align*}
	Similarly to proof of \eqref{xpartest}, applying Theorem \ref{key-estimate} to ${\bf (i)}$, ${\bf (ii)}$ and ${\bf (iii)}$ yields
	\begin{align}\label{ypartest}
	{\bf  O}\prec\check{c}_2({ B}_y^{\rm T}+{ B}_y)=\check{c}_2{\bf  Q}^{\rm T}(\tilde{ B}_y^{\rm T}+\tilde{ B}_y){\bf  Q}\preceq{ B}_y^{\rm T}{ D}_a+{ D}_a{ B}_y&\preceq\hat{c}_2{\bf  Q}^{\rm T}(\tilde{ B}_y^{\rm T}+\tilde{ B}_y){\bf  Q}\notag\\
	&=\hat{c}_2({ B}_y^{\rm T}+{ B}_y).
	\end{align}
	Moreover, it is easy to see that
	\begin{equation}\label{commutpart}
	\check{a}^2{ W}^{\rm T}{ W}\preceq{ W}^{\rm T}{ D}_a^2{ W}\preceq\hat{a}^2{ W}^{\rm T}{ W}.
	\end{equation}
	By \eqref{xpartest}--\eqref{commutpart},
	\begin{align}
	{\bf  O}&\prec{ I}_{\hat{M}}+\nu_1\check{c}_1({ B}_x^{\rm T}+{ B}_x)+\nu_2\check{c}_2({ B}_y^{\rm T}+{ B}_y)+\check{a}^2{ W}^{\rm T}{ W}\notag\\
	&\preceq{\bf A}^{\rm T}{\bf A}\notag\\
	&\preceq{ I}_{\hat{M}}+\nu_1\hat{c}_1({ B}_x^{\rm T}+{ B}_x)+\nu_2\hat{c}_2({ B}_y^{\rm T}+{ B}_y)+\hat{a}^2{ W}^{\rm T}{ W}.\label{2datracontrl}
	\end{align}
	For any non-zero vector ${ y}\in\mathbb{R}^{M\times 1}$, denote ${ z}={\bf P}^{-1}{ y}$. Then, it holds
	\begin{equation*}
	\frac{{ y}^{\rm T}({\bf A}{\bf P}^{-1})^{\rm T}({\bf A}{\bf P}^{-1}){ y}}{{ y}^{\rm T}{ y}}=\frac{{ z}^{\rm T}{\bf A}^{\rm T}{\bf A}{ z}}{{ z}^{\rm T}{\bf P}^{\rm T}{\bf P}{ z}}.
	\end{equation*}
	By \eqref{2datracontrl},
	\begin{align}\label{nrangeesti1}
	0&<\frac{{ z}^{\rm T}[{ I}_{\hat{M}}+\nu_1\check{c}_1({ B}_x^{\rm T}+{ B}_x)+\nu_2\check{c}_2({ B}_y^{\rm T}+{ B}_y)+\check{a}^2{ W}^{\rm T}{ W}]{ z}}{{ z}^{\rm T}[{ I}_{\hat{M}}+\nu_1\bar{a}({ B}_x^{\rm T}+{ B}_x)+\nu_2\bar{a}({ B}_y^{\rm T}+{ B}_y)+\bar{a}^2{ W}^{\rm T}{ W}]{ z}}\notag\\
	&\leq\frac{{ z}^{\rm T}{\bf A}^{\rm T}{\bf A}{ z}}{{ z}^{\rm T}{\bf P}^{\rm T}{\bf P}{ z}}\notag\\
	&\leq\frac{{ z}^{\rm T}[{ I}_{\hat{M}}+\nu_1\hat{c}_1({ B}_x^{\rm T}+{ B}_x)+\nu_2\hat{c}_2({ B}_y^{\rm T}+{ B}_y)+\hat{a}^2{ W}^{\rm T}{ W}]{ z}}{{ z}^{\rm T}[{ I}_{\hat{M}}+\nu_1\bar{a}({ B}_x^{\rm T}+{ B}_x)+\nu_2\bar{a}({ B}_y^{\rm T}+{ B}_y)+\bar{a}^2{ W}^{\rm T}{ W}]{ z}}.
	\end{align}
	By Proposition \eqref{wghtsumbdlem} and \eqref{nrangeesti1},
	\begin{equation*}
	\check{s}\leq\min\left\{\frac{\check{c}_1}{\bar{a}},\frac{\check{c}_2}{\bar{a}},\frac{\check{a}^2}{\bar{a}^2}\right\}\leq\frac{{ z}^{\rm T}{\bf A}^{\rm T}{\bf A}{ z}}{{ z}^{\rm T}{\bf P}^{\rm T}{\bf P}{ z}}\leq\max\left\{\frac{\hat{c}_1}{\bar{a}},\frac{\hat{c}_2}{\bar{a}},\frac{\hat{a}^2}{\bar{a}^2}\right\}\leq\hat{s}.
	\end{equation*}
	During the proof above, there is no constraint on $M$ and $N$. Thus, for any $N\geq 1$, any $M\geq1$, $\Sigma^2({\bf A}{\bf P}^{-1})\subset[\check{s},\hat{s}]$ and $\sup\limits_{N,M\geq 1}\cond({\bf A}{\bf P}^{-1})\leq\sqrt{\hat{s}/\check{s}}$.
\end{proof}
\section{Numerical experiments}\label{numerical}
In this section, we test several examples to support theoretical results of Theorems \ref{convergence1D}, \ref{convergence2D-1} and to show the efficiency of the Toeplitz preconditioner. We employ generalized minimal residual (PGMRES) method with the Toeplitz preconditioner to solve \eqref{matrixform} and \eqref{2Dmatrixform}. We denote PGMRES method with Toeplitz preconditioner by PGMRES-{\bf T}. The stopping criterion for PGMRES-{\bf T} is set as $\frac{||{\bf r}_k||_2}{||{\bf r}_0||_2}\leq$1e-7, where ${\bf r}_k$ denotes the residual vector at $k$-th iteration.
Also, to illustrate the efficiency of PGMRES-{\bf T}, we compare it with the direct solver, pivoted LU factorization (PLU).
All numerical experiments are performed via MATLAB R2015a on a PC with the configuration: Intel(R) Core(TM) i7-4720 CPU 2.60 GHz and 8 GB RAM.

Recall that $h$ is the spatial step-size for one-dimensional discretization. We also set $h_1=h_2=h$ in two-dimensional discretization for the related experiments in this section. Define the error as
$$E(h,\tau)=\max_{0\leq n\leq N}\|e^n\|.$$ Then, the spatial and temporal convergence rates are measured as follows
$$Rate_h=\log_2\bigg(\dfrac{E(2h,\tau)}{E(h,\tau)}\bigg),\quad Rate_\tau=\log_2\bigg(\dfrac{E_2(h,2\tau)}{E(h,\tau)}\bigg).$$

Denote by CPU, the running time by unit seconds. Denote by 'iter', the average of iteration numbers of PGMRES method for the $N$ linear systems in \eqref{matrixform} or \eqref{2Dmatrixform}.
\begin{example}\label{ex1}{\rm
		Consider a one-dimensional OSFDE with $[x_L,x_R]=[0,1]$, $T=1$ and
		\begin{align*}
		&d(x)=\cos(\pi x/2)+0.1,\\
		&f(x,t)=192x^3(1-x)^3t^2-2^6t^3d(x)\sum\limits_{k=3}^{6}\frac{\binom{3}{k-3}k!x^{k-\alpha}}{(-1)^{k-1}\Gamma(k+1-\alpha)}.
		\end{align*}
		The explicit expression of exact solution for the example is $u(x,t)=2^6x^3(1-x)^3t^3$.
		
		We employ both PGMRES-{\bf T} and PLU to solve the linear systems \eqref{1dlinsys} arising from Example \ref{ex1}, the results of which are listed in Tables \ref{ex1fixedN}--\ref{ex1fixedM}.
		
		From Tables \ref{ex1fixedN}--\ref{ex1fixedM}, we see that the CPU cost of PGMRES-{\bf T} is much less than that of PLU solver while the error, $E(h,\tau)$ of the two solvers are almost the same, which demonstrate the efficiency of the Toeplitz preconditioner. Also, as $\tau$ or $h$ changes in Tables \ref{ex1fixedN}--\ref{ex1fixedM}, the iteration number of PGMRES-{\bf T} varies slightly, which shows a linear convergence of PGMRES-{\bf T}. Moreover, the temporal convergence rate, $Rate_h$ and the spatial convergence rate,$Rate_{\tau}$ from Tables \ref{ex1fixedN}--\eqref{ex1fixedM} are always close to 2, which supports the theoretical result of Theorem \ref{convergence1D}.
		\begin{table}[H]
			\begin{center}
				\caption{Numerical results for Example \ref{ex1} when $\tau=2^{-10}$.				
				}\label{ex1fixedN}
				\setlength{\tabcolsep}{0.6em}
				\begin{tabular}[c]{cc|cccc|ccc}
					\hline
					~&~& \multicolumn{4}{c|}{PGMRES-{\bf T}}& \multicolumn{3}{c}{PLU}\\
					\cline{3-9}
					$\alpha$&$h$&$\mathrm{iter}$&CPU&$E(h,\tau)$&$Rate_h$&CPU&$E(h,\tau)$&$Rate_h$\\
					\hline
					1.2       &$2^{-8}$ &2.1  &3.12s  &3.47e-5      &--      &48.40s   &3.49e-5&--      \\
					~         &$2^{-9}$ &2.2  &6.27s  &8.42e-6      &2.04    &121.34s  &8.60e-6&2.02    \\
					~         &$2^{-10}$&2.3  &12.07s &1.87e-6      &2.17    &327.90s  &1.99e-6&2.11    \\
					\hline
					1.5       &$2^{-8}$ &3.4  &4.12s  &3.18e-5      &--      &48.79s   &3.18e-5&--      \\
					~         &$2^{-9}$ &3.6  &7.89s  &7.82e-6      &2.03    &123.67s  &7.81e-6&2.02    \\
					~         &$2^{-10}$&3.9  &15.31s &1.66e-6      &2.23    &350.71s  &1.81e-6&2.11    \\
					\hline
					1.8       &$2^{-8}$ &4.5  &5.11s  &2.51e-5      &--      &53.24s   &2.50e-5&--      \\
					~         &$2^{-9}$ &4.7  &9.51s  &6.23e-6      &2.01    &138.52s  &6.12e-6&2.03    \\
					~         &$2^{-10}$&4.9  &17.37s &1.57e-6      &1.99    &350.04s  &1.40e-6&2.13    \\
					\hline			
				\end{tabular}
			\end{center}
		\end{table}
		
		\begin{table}[H]
			\begin{center}
				\caption{Numerical results for Example \ref{ex1} when $\tau=2^{-11}$.				
				}\label{ex1fixedM}
				\setlength{\tabcolsep}{0.6em}
				\begin{tabular}[c]{cc|cccc|ccc}
					\hline
					~&~& \multicolumn{4}{c|}{PGMRES-{\bf T}}& \multicolumn{3}{c}{PLU}\\
					\cline{3-9}
					$\alpha$&$\tau$&$\mathrm{iter}$&CPU&$E(h,\tau)$&$Rate_{\tau}$&CPU&$E(h,\tau)$&$Rate_{\tau}$\\
					\hline
					1.2       &$2^{-7}$&10.6 &4.81s  &2.01e-5      &--      &120.74s  &2.01e-5&--      \\
					~         &$2^{-8}$&6.7  &5.63s  &4.80e-6      &2.07    &252.02s  &4.74e-6&2.08    \\
					~         &$2^{-9}$&3.8  &8.20s  &1.16e-6      &2.05    &493.55s  &9.68e-7&2.29    \\
					\hline
					1.5       &$2^{-7}$&12.4 &4.47s  &2.17e-5      &--      &124.21s  &2.17e-5&--      \\
					~         &$2^{-8}$&9.5  &6.30s  &5.19e-6      &2.06    &252.72s  &5.19e-6&2.06    \\
					~         &$2^{-9}$&6.5  &9.27s  &1.08e-6      &2.27    &486.69s  &1.12e-6&2.22    \\
					\hline
					1.8       &$2^{-7}$&12.6 &4.55s  &2.34e-5      &--      &119.27s  &2.34e-5&--      \\
					~         &$2^{-8}$&10.1 &6.47s  &5.73e-6      &2.03    &238.38s  &5.68e-6&2.04    \\
					~         &$2^{-9}$&7.4  &9.67s  &1.45e-6      &1.98    &478.66s  &1.28e-6&2.15    \\
					\hline			
				\end{tabular}
			\end{center}
		\end{table}
	}
\end{example}

\begin{example}\label{ex2}
	{\rm
		Consider a two-dimensional OSFDE with $ [x_L,x_R]=[y_L,y_R]=[0,2]$, $T=1$ and
		\begin{align*}
		&d(x,y)=x^2+y^2+20,\quad e(x,y)=\sin\left[ \frac{\pi}{24}(x+4)\right]+\sin\left[ \frac{\pi}{24}(y+4)\right],\\
		&f(x,y,t)=3x^4(2-x)^4y^4(2-y)^4t^2-t^3y^4(2-y)^4d(x,y)\sum\limits_{k=4}^{8}\frac{\binom{4}{k-4}2^{8-k}k!x^{k-\alpha}}{(-1)^k\Gamma(k+1-\alpha)}\\
		&\qquad\qquad\qquad\qquad\qquad\qquad\qquad\quad~-t^3x^4(2-x)^4e(x,y)\sum\limits_{k=4}^{8}\frac{\binom{4}{k-4}2^{8-k}k!y^{k-\beta}}{(-1)^k\Gamma(k+1-\beta)}.
		\end{align*}
		The explicit expression of exact solution for the example is $u(x,y,t)=x^4(2-x)^4y^4(2-y)^4t^3$.
		
		We employ both PGMRES-{\bf T} and PLU to solve the linear systems \eqref{2dlinsys} arising from Example \ref{ex2}, the results of which are listed in Tables \ref{ex2fixedN}--\ref{ex2fixedM}.
		
		Tables \ref{ex2fixedN}--\ref{ex2fixedM} shows that the CPU cost of PGMRES-{\bf T} is much less than that of PLU, which demonstrates the efficiency of the Toeplitz preconditioner in two-dimensional case. Again, the iteration number of PGMRES-{\bf T} changes slightly as $\tau$ or $h$ changes in Tables \ref{ex2fixedN}--\ref{ex2fixedM}, which shows a linear convergence of PGMRES-{\bf T}. Moreover, the convergence rates, $Rate_h$ and $Rate_{\tau}$ shown in Tables \ref{ex2fixedN}--\ref{ex2fixedM} are close to 2, which coincides with the theoretical result of Theorem \ref{convergence2D-1}.
		\begin{table}[H]
			\begin{center}
				\caption{Numerical results for Example \ref{ex2} when $\tau=2^{-7}$.				
				}\label{ex2fixedN}
				\setlength{\tabcolsep}{0.5em}
				\begin{tabular}[c]{cc|cccc|ccc}
					\hline
					~&~& \multicolumn{4}{c|}{PGMRES-{\bf T}}& \multicolumn{3}{c}{PLU}\\
					\cline{3-9}
					$(\alpha,\beta)$&$h$&$\mathrm{iter}$&CPU&$E(h,\tau)$&$Rate_h$&CPU&$E(h,\tau)$&$Rate_h$\\
					\hline
					(1.01,1.09) &$2^{-4}$ &4.0  &6.63s  &3.38e-3      &--      &35.73s   &3.38e-3&--      \\
					~           &$2^{-5}$ &4.1  &9.80s  &8.14e-4      &2.05    &234.97s  &8.14e-4&2.05    \\
					~           &$2^{-6}$ &4.1  &16.19s &1.90e-4      &2.10    &5017.20s &1.90e-4&2.10    \\
					\hline
					(1.5,1.3) &$2^{-4}$ &4.4  &7.43s  &3.00e-3      &--      &34.20s   &3.00e-3&--      \\
					~         &$2^{-5}$ &4.6  &10.62s &7.34e-4      &2.03    &219.68s  &7.34e-4&2.03    \\
					~         &$2^{-6}$ &5.8  &20.95s &1.72e-4      &2.09    &4875.58s &1.72e-4&2.09    \\
					\hline
					(1.5,1.6) &$2^{-4}$ &4.9  &7.59s  &3.00e-3      &--      &34.00s   &3.00e-3&--      \\
					~         &$2^{-5}$ &5.8  &12.71s &7.41e-4      &2.03    &220.50s  &7.41e-4&2.03    \\
					~         &$2^{-6}$ &6.8  &23.99s &1.73e-4      &2.09    &4769.31s &1.73e-4&2.10    \\
					\hline
					(1.5,1.9) &$2^{-4}$ &5.8  &8.75s  &3.00e-3      &--      &33.81s   &3.00e-3&--      \\
					~         &$2^{-5}$ &6.7  &14.47s &7.33e-4      &2.03    &219.61s  &7.33e-4&2.03    \\
					~         &$2^{-6}$ &7.4  &25.85s &1.71e-4      &2.10    &4746.60s &1.71e-4&2.10    \\
					\hline
					(1.2,1.2) &$2^{-4}$ &4.1  &6.63s  &3.30e-3      &--      &33.50s   &3.30e-3&--      \\
					~         &$2^{-5}$ &4.2  &9.71s  &8.03e-4      &2.03    &217.33s  &8.03e-4&2.03    \\
					~         &$2^{-6}$ &4.6  &17.13s &1.89e-4      &2.09    &4653.91s &1.89e-4&2.09    \\
					\hline
					(1.5,1.5) &$2^{-4}$ &4.5  &7.13s  &3.00e-3      &--      &33.71s   &3.00e-3&--      \\
					~         &$2^{-5}$ &5.4  &12.09s &7.40e-4      &2.03    &218.36s  &7.40e-4&2.03    \\
					~         &$2^{-6}$ &6.5  &23.10s &1.73e-4      &2.10    &4663.71s &1.73e-4&2.10    \\
					\hline
					(1.8,1.8) &$2^{-4}$ &6.2  &9.22s  &2.30e-3      &--      &33.73s   &2.30e-3&--      \\
					~         &$2^{-5}$ &7.2  &15.40s &5.71e-3      &2.04    &218.98s  &5.71e-4&2.04    \\
					~         &$2^{-6}$ &8.3  &28.54s &1.31e-4      &2.12    &4841.52s &1.31e-4&2.12    \\
					\hline			
				\end{tabular}
			\end{center}
		\end{table}
		
		\begin{table}[H]
			\begin{center}
				\caption{Numerical results for Example \ref{ex2} when $h=2^{-6}$.				
				}\label{ex2fixedM}
				\setlength{\tabcolsep}{0.5em}
				\begin{tabular}[c]{cc|cccc|ccc}
					\hline
					~&~& \multicolumn{4}{c|}{PGMRES-{\bf T}}& \multicolumn{3}{c}{PLU}\\
					\cline{3-9}
					$(\alpha,\beta)$&$\tau$&$\mathrm{iter}$&CPU&$E(h,\tau)$&$Rate_h$&CPU&$E(h,\tau)$&$Rate_h$\\
					\hline
					(1.01,1.09) &$2^{-3}$ &219.1&62.13s &6.77e-3      &--      &334.49s  &6.77e-3&--      \\
					~           &$2^{-4}$ &82.9 &43.48s &1.61e-3      &2.07    &636.48s  &1.61e-3&2.07    \\
					~           &$2^{-5}$ &33.1 &30.14s &3.48e-4      &2.21    &1249.74s &3.48e-4&2.21    \\
					\hline
					(1.5,1.3) &$2^{-3}$ &14.5 &3.50s  &6.80e-3      &--      &338.66s  &6.80e-3&--      \\
					~         &$2^{-4}$ &12.7 &5.71s  &1.60e-3      &2.07    &612.05s  &1.60e-3&2.07    \\
					~         &$2^{-5}$ &10.2 &9.00s  &3.50e-4      &2.21    &1226.04s &3.50e-4&2.21    \\
					\hline
					(1.5,1.6) &$2^{-3}$ &12.3 &3.01s  &6.80e-3      &--      &331.95s  &6.80e-3&--      \\
					~         &$2^{-4}$ &11.3 &5.16s  &1.60e-3      &2.07    &644.72s  &1.60e-3&2.07    \\
					~         &$2^{-5}$ &10.1 &8.96s  &3.45e-4      &2.23    &1220.89s &3.45e-4&2.23    \\
					\hline
					(1.5,1.9) &$2^{-3}$ &11.0 &2.73s  &6.80e-3      &--      &336.93s  &6.80e-3&--      \\
					~         &$2^{-4}$ &10.3 &4.77s  &1.60e-3      &2.07    &638.47s  &1.60e-3&2.07    \\
					~         &$2^{-5}$ &9.5  &8.42s  &3.46e-4      &2.23    &1217.05s &3.46e-4&2.23    \\
					\hline
					(1.2,1.2) &$2^{-3}$ &15.4 &3.63s  &6.80e-3      &--      &338.42s  &6.80e-3&--      \\
					~         &$2^{-4}$ &12.7 &5.70s  &1.60e-3      &2.07    &629.27s  &1.60e-3&2.07    \\
					~         &$2^{-5}$ &10.2 &9.01s  &3.47e-4      &2.22    &1238.65s &3.47e-4&2.22    \\
					\hline
					(1.5,1.5) &$2^{-3}$ &13.4 &3.23s  &6.80e-3      &--      &334.47s  &6.80e-3&--      \\
					~         &$2^{-4}$ &12.1 &5.47s  &1.60e-3      &2.07    &646.24s  &1.60e-3&2.07    \\
					~         &$2^{-5}$ &10.5 &9.19s  &3.46e-4      &2.22    &1223.25s &3.46e-4&2.22    \\
					\hline
					(1.8,1.8) &$2^{-3}$ &12.3 &3.00s  &6.80e-3      &--      &335.93s  &6.80e-3&--      \\
					~         &$2^{-4}$ &11.5 &5.23s  &1.60e-3      &2.06    &625.35s  &1.60e-3&2.06    \\
					~         &$2^{-5}$ &10.6 &9.30s  &3.58e-4      &2.20    &1221.17s &3.58e-4&2.20    \\
					\hline			
				\end{tabular}
			\end{center}
		\end{table}	
	}
\end{example}

%Counterexample:
%\begin{example}\label{ex3}
%We replace the $d(x,y)$ and $e(x,y)$ in Example \ref{ex3} with $d(x,y)=\cos(35\pi x)+1.01$, $e(x,y)=\sin(35\pi y)+1.01$ respectively, and the corresponding forcing term $f(x,y,t)$ is chosen to such that $u(x,y,t)=x^4(2-x)^4y^4(2-y)^4t^3$ is still the solution.
%\end{example}

\section{Concluding remarks}\label{Concluding}

We study second-order schemes for time-dependent one- and two-dimensional OSFDEs with variable diffusion coefficients, in which implicit Crank-Nicolson scheme and WSGD formula are employed to discretize the temporal and the spatial derivatives, respectively. Theoretically, we have established the unconditional stability and second-order convergence for the one-dimensional scheme without additional assumption, and for the two-dimensional scheme with certain assumptions on diffusion coefficients presented in Corollaries \ref{case1lemma}--\ref{case2lemma}. To accelerate the solution process, Toeplitz preconditioners have been proposed for both one- and two-dimensional schemes. The condition numbers of the preconditioned matrices have been proven to be bounded by a constant independent of discretization step-sizes under certain assumptions on the diffusion coefficients presented in Theorems \ref{1dcondesti}, \ref{1dcondesti1}, \ref{2dclstthm}. Numerical results reported have shown the second-order convergence rate of the proposed schemes and the efficiency of the proposed preconditioners.

\end{document}